\documentclass{amsart}
\usepackage{verbatim}
\usepackage{amsmath,amssymb,amsthm}

\newtheorem{theorem}[subsection]{Theorem}
\newtheorem{lemma}[subsection]{Lemma}
\newtheorem{corollary}[subsection]{Corollary}
\newtheorem{proposition}[subsection]{Proposition}

\theoremstyle{definition}
\newtheorem{definition}[subsubsection]{Definition}
\newtheorem{remark}[subsection]{Remark}

\newcommand{\dist}{\mathrm{dist}}
\newcommand{\tr}{\mathrm{tr}}

\newcommand{\spt}{{\mathrm{spt}\,}}

\newcommand{\bA}{\bar{A}}
\newcommand{\bH}{\bar{H}}
\newcommand{\bh}{\bar{h}}

\newcommand{\tA}{\tilde{A}}
\newcommand{\tH}{\tilde{H}}
\newcommand{\tS}{\tilde{S}}
\newcommand{\tQ}{\tilde{Q}}
\newcommand{\tlambda}{\tilde{\lambda}}
\newcommand{\blambda}{\bar{\lambda}}

\newcommand{\vecdiv}{\mathrm{div}}

\title{Convexity estimates for mean curvature flow with free boundary}

\author{Nick Edelen}
\address{Department of Mathematics, Stanford University, 450 Serra Mall, Bldg 380, CA 94305}
\email{nedelen@math.stanford.edu}
\keywords{mean curvature flow; convexity estimates; free boundary; limit flows}

\begin{document}

\begin{abstract}
We prove the estimates of \cite{huisken-sinestrari:scalar-pinching} and \cite{huisken-sinestrari:convex-pinching} for finite-time singularities of mean-convex, mean curvature flow with free boundary in a barrier $S$.  Here $S$ can be any embedded, oriented surface in $R^{n+1}$ of bounded geometry and positive inscribed radius.  We also prove the estimate \cite{huisken:umbilic-pinching} in the case of convex flows and $S = S^n$, which gives an alternative proof to \cite{stahl:singularity}.
\end{abstract}

\maketitle

\section{Introduction}

We are interested in immersed, mean-convex, mean curvature flow with free boundary in a surface $S$.  We reprove the estimates in \cite{huisken-sinestrari:scalar-pinching}, and \cite{huisken-sinestrari:convex-pinching} for this class of flows.  These provide very direct, general pinching results for limit flows at singularities, and require no embeddedness or curvature assumptions.  We further prove the estimates in \cite{huisken:umbilic-pinching} when $S$ is the sphere.

Consider a smooth, embedded, oriented hypersurface $S \subset R^{n+1}$, with choice of normal $\nu_S$, having bounded geometry and positive inscribed radius.  We refer to $S$ as the \emph{barrier surface}.  If $\Sigma^n \subset R^{n+1}$ is a compact, mean-convex hypersurface with boundary, we say \emph{$\Sigma$ meets $S$ orthogonally} if $\partial \Sigma \subset S$, and the outer normal of $\partial\Sigma \subset \Sigma$ coincides with $\nu_S$.


Let $\Sigma_0 = \Sigma$ meet the barrier $S$ orthogonally.  Then the mean curvature flow of $\Sigma_0$, with free-boundary in $S$, is a family of immersions $F_t : \Sigma_0 \times [0, T) \to R^{n+1}$ such that
\begin{align*}
&\partial_t F_t = -H \nu , \text{ for all } p \in \Sigma, t > 0 \\
&\text{$F_t(\Sigma)$ meets $S$ orthogonally for all $t \geq 0$} \\
&F_0 \equiv \mathrm{Id}_{\Sigma_0} .
\end{align*}
Here $H$ is the mean curvature, and $\nu$ the outer normal, oriented so that $\mathbf{H} = -H \nu$ is the mean curvature vector.  We often write $\Sigma_t = F_t(\Sigma)$, and will identify and surface and its immersion.

It was shown by Stahl \cite{stahl:regularity} that the mean curvature flow with free-boundary in $S$ always exists on some maximal time interval $[0, T)$, for $T \leq \infty$, such that if $T < \infty$ then necessarily $\max_{\Sigma_t} |A| \to \infty$ as $t \to T$.  Here $|A|$ is the norm of the second fundamental form $A$.

Type-I tangent flows of mean curvature flow with free boundary have been classified by Buckland \cite{buckland}.  Our convexity estimates work towards classifying type-II limit flows with free boundary.  Stahl \cite{stahl:singularity} has shown Theorem \ref{theorem:umbilic-pinching} using a different method.

We prove the following theorems concerning the mean curvature flow of $\Sigma_0$ with free-boundary in $S$.  Throughout the duration of this paper we assume $\Sigma_0$ is compact, mean-convex.

\begin{theorem}\label{theorem:A-H-bound-intro}
There are constants $\alpha = \alpha(S) \geq 0$ and $C = C(S, \Sigma_0)$ so that
\begin{equation}
\max_{\Sigma_t} \frac{|A|}{H} \leq C e^{\alpha t}
\end{equation}
for all time of existence.  In particular, if $T < \infty$, then
\[
|A| \leq C(S, \Sigma_0, T) H.
\]
\end{theorem}

\begin{definition}
Given a vector $\mu \in R^n$, and $k \in \{1, \ldots, n\}$, we let
\begin{equation*}
s_k(\mu) = \sum_{1\leq i_1 < \ldots < i_k \leq n} \mu_{i_1} \cdots \mu_{i_k}
\end{equation*}
be the $k$-th symmetric polynomial of $\mu$.  We adopt the convention that $s_0 \equiv 1$.  If $s_{k-1}(\mu) \neq 0$, we let
\begin{equation*}
q_k(\mu) = \frac{s_k(\mu)}{s_{k-1}(\mu)} .
\end{equation*}

Given a real symmetric $n \times n$ matrix $M$, define $s_k(M) = s_k(\mu)$ where $\mu \in R^n$ is the vector of eigenvalues of $M$.  Similarly, where possible set $q_k(M) = q_k(\mu)$.  Notice that $s_k$ is a polynomial in the entries of $M$.

Given a surface $\Sigma$, define the smooth function $S_k$ by
\[
S_k(p) = s_k(A(p)) = s_k(\lambda(p))
\]
where $\lambda$ the vector of principle curvatures.  Similarly where possible set $Q_k = q_k(A)$.  We have that $H \equiv S_1$, and $|A|^2 \equiv S_1^2 - 2S_2$.
\end{definition}

\begin{theorem}[Convexity pinching]\label{theorem:convexity-pinching}
If $T < \infty$, then for any $k \in \{1, \ldots, n\}$, $\eta > 0$, there is a constant $C = C(S, \Sigma_0, T, \eta, n)$ such that
\[
S_k \geq -\eta H^k - C
\]
at all points in spacetime.
\end{theorem}

For $T < \infty$, by rescaling $\Sigma_t$ along an essential blow-up sequence (c.f. Section 4 of \cite{huisken-sinestrari:scalar-pinching}), we obtain an eternal limit flow $\tilde \Sigma_\tau$ with free boundary in a hyperplane.  This can be reflected to a mean curvature flow without boundary.  Theorem 4.1 of \cite{huisken-sinestrari:convex-pinching} therefore proves the following Corollary of Theorem \ref{theorem:convexity-pinching}.

\begin{corollary}
If $T < \infty$, then any limit flow of $\Sigma_t$ at a type-II singularity is a weakly convex flow $\tilde \Sigma_\tau$ with free boundary in a hyperplane.  After reflection to a flow without boundary, $\tilde\Sigma_\tau$ is a convex translating soliton.  Further, we can write $\tilde \Sigma_\tau = R^{n-k} \times \tilde \Sigma^k_\tau$, where $\tilde\Sigma^k_\tau$ is strictly convex.
\end{corollary}

\begin{remark}
By our assumptions on $S$, $\Sigma_t$ can only move a finite distance in finite time (see Proposition \ref{prop:finite-dist}).  In other words, if $T < \infty$, $\Sigma_t$ approaches a set-theoretic limit.  One can probably construct examples with free boundary in a barrier of unbounded geometry/zero inscribed radius which shoot off to infinity in finite time.

If $T = \infty$ and $\Sigma_t$ stays in a bounded region, then either $|\Sigma_t| \to 0$ or by a standard argument $\Sigma_t$ approaches a minimal surface.
\end{remark}

\begin{remark}
Theorems \ref{theorem:A-H-bound-intro} and \ref{theorem:convexity-pinching} also hold in a Riemannian manifold of bounded geometry.  In fact the error terms introduced are almost entirely subsumed by the perturbations we already make.
\end{remark}

\begin{theorem}[Umbilic pinching, \cite{stahl:singularity}]\label{theorem:umbilic-pinching}
If $\Sigma_0$ is convex and $S = S^n$, then $\Sigma_t$ shrinks to a point in finite time, and any limit flow at the singularity is umbilic.  In particular, there is a sequence of rescalings which converge to a shrinking half-sphere with free-boundary in a hyperplane.
\end{theorem}

The following notation is used extensively.

\begin{definition}
Writing $f = O(g)$ means there is a constant $c = c(n, S)$ such that $|f| \leq c |g|$.
\end{definition}

We outline our approach.  The main obstruction to analyzing free boundary behavior in a general barrier $S$ is obtaining boundary conditions on $|A|$, or $S_k$ when $k > 1$.  We perturb the second fundamental form so that the normal $\nu_S$ is an eigenvector, which allows us to obtain boundary conditions on the perturbed principle curvatures.

This introduces relatively large error terms into the evolution equations of the perturbed $|\bar A|$ and $\bar H$.  The error is too large to naively give exponential behavior of the quantity $|\bar A|/H$.  To handle this, and to correct the boundary behavior, in proving Theorem \ref{theorem:A-H-bound-intro} we must consider instead the evolution of
\begin{equation}\label{eqn:fixed-evolution}
\frac{|\bar A| + a}{H} \phi ,
\end{equation}
for some large constant $a$, and barrier function $\phi$.

The evolution equation for \eqref{eqn:fixed-evolution} will have the right form except for a gradient term resulting from $\phi$.  To control bad gradient terms we observe that by restricting to points where $|\bar A| \geq 2\bar H$, we can squeeze a term out of Cauchy's inequality:
\[
|\nabla \bar A|^2 - |\nabla |\bar A||^ 2\geq \frac{1}{c} |\nabla \bar A|^2 + O(|\bar A|^2) .
\]

Given Theorem \ref{theorem:A-H-bound-intro}, we can adapt the Stampacchia iteration scheme used by \cite{huisken-sinestrari:convex-pinching} to prove Theorems \ref{theorem:convexity-pinching} and \ref{theorem:umbilic-pinching}.  The key step is proving a trace-like formula for free boundary surfaces.  The argument is sufficiently robust to handle without problem the perturbation terms.

I am very grateful to my advisor Simon Brendle for his guidance and encouragement, Brian White for many illuminating discussions, and Otis Chodosh for his support and advice.  I also thank Robert Haslhofer and Gerhard Huisken for helpful conversations.  This work was partially supported by the Royden fellowship.  Some of this work was also completed while visiting Columbia University, and I'm grateful for their hospitality.

\section{Michael-Simon with (free) boundary}
\label{section:michael-simon}

We adapt the Michael-Simon inequality \cite{michael-simon} to surfaces with smooth boundary, and surfaces meeting a barrier surface orthogonally.

\begin{lemma}\label{lemma:trace-formula}
There is a constant $c = c(n, S)$ such that for any $\Sigma$ meeting $S$ orthogonally, and any $v \in C^1(\bar\Sigma)$, 
\begin{equation}
\frac{1}{c} \int_{\partial\Sigma} |v| \leq \int_\Sigma |\nabla v| + \int_\Sigma |Hv| + \int_\Sigma |v| .
\end{equation}
\end{lemma}

\begin{proof}
Choose (and fix, for the duration of the paper) a smooth vector field $X$ on $R^{n+1}$ which is $0$ outside a neighborhood of $S$, and $X \equiv \nu_S$ on $S$.  Then
\begin{align*}
\int_{\partial\Sigma} |v|
&= \int_{\partial\Sigma} |v| X \cdot \nu \\
&= \int_\Sigma \vecdiv_\Sigma(|v| X^T) \\
&= \int_\Sigma \nabla |v| \cdot X + |v| \vecdiv_\Sigma(X) - |v| X \cdot \nu H \\
&\leq \max |X| \int |\nabla v| + n\max |\nabla X| \int |v| + \max |X| \int |v H| \qedhere
\end{align*}
\end{proof}

\begin{theorem}\label{theorem:micheal-simon-boundary}
There is a constant $c = c(n)$ such that for any $v \in C^1_c(\bar\Sigma)$, we have
\begin{equation}
\frac{1}{c} \left( \int_\Sigma |v|^{\frac{n}{n-1}} \right)^{\frac{n-1}{n}} \leq \int_\Sigma |\nabla v| + \int_\Sigma |Hv| + \int_{\partial\Sigma} |v| .
\end{equation}
\end{theorem}

\begin{proof}
By replacing $v$ with $|v|$ we can without loss of generality suppose $v \geq 0$.

For $x \in \partial\Sigma$, let $\gamma_x(t)$ be the unit speed geodesic in $\Sigma$ with initial conditions $\gamma_x(0) = x$ and $\gamma'_x(0) \perp \partial\Sigma$.  For sufficiently small $\epsilon$, depending only on the curvatures of $\Sigma$ and $\partial\Sigma \subset \Sigma$, the function $\phi : [0, \epsilon] \times \partial \Sigma \to \Sigma$ mapping $(t, x) \mapsto \gamma_x(t)$ is a diffeomorphism, with its Jacobian bounded like $|J\phi| \in [\frac{1}{2}, 2]$.

We deduce, for any $\epsilon$ sufficiently small, 
\begin{align}
\int_{\dist(\cdot, \partial\Sigma) \leq \epsilon} v 
&= \int_0^\epsilon \int_{\partial\Sigma} v |J\phi| \nonumber \\
&\leq 2 \int_0^\epsilon \int_{\partial \Sigma} v(0, x) + 2 \int_0^\epsilon \int_{\partial\Sigma} t \frac{\partial v}{\partial t}(t^*(x), x) \nonumber \\
&\leq 2\epsilon \int_{\partial\Sigma} v + \epsilon^2 |\partial\Sigma| \sup_{\Sigma} |\nabla v| \label{eqn:ms-near-boundary}
\end{align} 
here $t^*(x) \in (0, \epsilon)$.

Now take $\eta$ a function which is $\equiv 1$ on $\dist(\cdot, \partial\Sigma) \geq \epsilon$ and $\equiv 0$ on $\partial \Sigma$, and such that $|\nabla \eta| \leq 2/\epsilon$.  From \eqref{eqn:ms-near-boundary}
\begin{align*}
\int_\Sigma ((1 - \eta) v)^{\frac{n}{n-1}} 
&\leq \int_{\dist(\cdot, \partial\Sigma) \leq \epsilon} v^{\frac{n}{n-1}} \\
&\leq 2\epsilon \int_{\partial \Sigma} v^{\frac{n}{n-1}} + \epsilon^2 |\partial\Sigma| \sup_{\Sigma} |\nabla (v^\frac{n}{n-1})| \\
&\leq \epsilon C
\end{align*}
for $C$ independant of $\epsilon$.

Therefore, using the Michel-Simon inequality and \eqref{eqn:ms-near-boundary} again, 
\begin{align*}
||v||_{\frac{n}{n-1}} 
&\leq ||\eta v||_{\frac{n}{n-1}} + ||(1-\eta)v||_{\frac{n}{n-1}} \\
&\leq c\int_\Sigma \eta |\nabla v| + c\int_\Sigma |H|\eta v + c\int_\Sigma |\nabla \eta| v + \epsilon^{\frac{n-1}{n}} C \\
&\leq c\int_\Sigma |\nabla v| + c\int_\Sigma |H|v + 2c/\epsilon \int_{\dist(\cdot, \partial\Sigma) \leq \epsilon} v + \epsilon^{1/2} C \\
&\leq c\int_\Sigma |\nabla v| + c\int_\Sigma |H|v + 4c \int_{\partial\Sigma} v + \epsilon |\partial\Sigma| \sup_{\Sigma} |\nabla v|  + \epsilon^{1/2} C
\end{align*}
for $c = c(n)$ and all $\epsilon$ sufficiently small.  Taking $\epsilon$ to $0$ proves the lemma.
\end{proof}

\begin{theorem}\label{theorem:micheal-simon-free-boundary}
If $\Sigma$ meets $S$ orthogonally, and $v \in C^1(\bar\Sigma)$, then for any $p < n$,
\begin{equation}
||v||_{\frac{np}{n - p}; \Sigma} \leq c(||\nabla v||_{p;\Sigma} + ||Hv||_{p;\Sigma} + ||v||_{p;\Sigma})
\end{equation}
where $c = c(n, p, S)$.
\end{theorem}

\begin{proof}
Combine Lemma \ref{lemma:trace-formula} and Theorem \ref{theorem:micheal-simon-boundary} to obtain the desired inequality with $p = 1$.  Then set $v = w^\gamma$ to obtain
\begin{align*}
\left( \int w^{\gamma \frac{n}{n-1}} \right)^{\frac{n-1}{n}} 
&\leq c\gamma \int w^{\gamma - 1} |\nabla w| + c \int w^{\gamma - 1} H w  + c \int w^{\gamma - 1} w \\
&\leq c \left( w^{(\gamma - 1) \frac{p}{p-1}} \right)^{\frac{p-1}{p}} ( ||\nabla w||_p + ||Hw||_p + ||w||_p).
\end{align*}
Now choose $\gamma$ such that
\[
\gamma \frac{n}{n-1} = (\gamma - 1) \frac{p}{p-1} \qedhere
\]
\end{proof}

\begin{corollary}\label{corollary:micheal-simon-n=2}
If $n = 2$, then for any $q \in (1, \infty)$, 
\begin{equation}
||v||_{2q; \Sigma} \leq c |\Sigma|^{\frac{1}{2q}} (||\nabla v||_{2; \Sigma} + ||Hv||_{2;\Sigma} + ||v||_{2;\Sigma})
\end{equation}
where $c = c(q, S)$.
\end{corollary}

\begin{proof}
Take $n = 2$ and $p = 2-\delta$ in Theorem \ref{theorem:micheal-simon-free-boundary}, for $\delta \in (0, 1)$.  Set $q = \frac{2-\delta}{\delta}$.  Then we have for any $r \in (1, \infty)$, 
\begin{align*}
||v||_{2q} 
&\leq c (||\nabla v||_{2-\delta} + ||Hv||_{2-\delta} + ||v||_{2-\delta}) \\
&\leq c |\Sigma|^{(1-1/r) \frac{1}{2-\delta}} (||\nabla v||_{ r(2-\delta)} + ||Hv||_{r(2-\delta)} + ||v||_{r(2-\delta)}) .
\end{align*}
Then set $r = \frac{2}{2-\delta}$.
\end{proof}

\begin{remark}
Since $|\Sigma_t|$ is monotone decreasing (Remark \ref{remark:area-decreasing}),
\[
c(q, S) |\Sigma_t|^{1/2q} \leq c(q, S, |\Sigma_0|) .
\]
\end{remark}


\section{General inequalities and Stampacchia iteration}
\label{section:stampacchia}

Each pinching result uses a Stampacchia iteration scheme to obtain pointwise bounds from $L^p$-bounds.  All cases can be handled by the following general principle.

Take $(\Sigma_t)_{t\in [0, T)}$ a mean curvature flow with free boundary in $S$, and assume $T < \infty$.  Let $f_\alpha$ be some non-negative function on $\Sigma_t$, depending on some parameters $\alpha = \alpha(S, \Sigma_0, T, n)$.  Let $\tilde G \geq 0$ and $\tH > 0$ be functions on $\Sigma_t$ such that
\[
H = O(\tH), \quad \nabla \tH = O(\tilde G) .
\]

Let $f = f_\alpha \tilde H^\sigma$, and $f_k = (f - k)_+$, where $\sigma > 0$ will be small and $k > 0$ large.  Write $A(k) = \{ f \geq k \}$, and $A(k, t) = A(k) \cap \Sigma_t$.

We say $f$ satisfies $(\star)$ if there are constants $c = c(S, \Sigma_0, T, n, \alpha)$, and $C = C(S, \Sigma_0, T, n, \alpha, p, \sigma)$, such that for any $p > p_0(n, \alpha, c)$, $\sigma < 1/2$, $k > 0$ and $\beta > 0$, the following two equations hold:

(POINCARE-LIKE)
\begin{align*}
\frac{1}{c} \int_{\Sigma_t} f^p \tilde H^2 
&\leq (p + p/\beta) \int_{\Sigma_t} f^{p-2} |\nabla f|^2 + (1 + \beta p) \int_{\Sigma_t} \frac{
\tilde G^2}{\tilde H^{2-\sigma}}f^{p-1} + \int_{\Sigma_t} f^p \\
&\quad + \int_{\partial\Sigma_t} f^{p-1}\tilde H^\sigma
\end{align*}

(EVOLUTION-LIKE)
\begin{align*}
\partial_t \int_{\Sigma_t} f^p_k 
&\leq -\frac{1}{3}p^2\int_{\Sigma_t} f^{p-2}_k|\nabla f|^2 -  p/c \int_{\Sigma_t} \frac{\tilde G^2}{\tilde H^{2-\sigma}} f^{p-1}_k + cp\sigma \int_{A(k, t)} \tilde H^2 f^p \\
&\quad - 1/5 \int_{\Sigma_t} \tilde H^2 f_k^p + C \int_{A(k,t)} f^p + C|A(k)| + cp \int_{\partial\Sigma_t} f^{p-1}_k\tilde H^\sigma
\end{align*}

This section culminates in proving
\begin{theorem}\label{theorem:f-is-bounded}
If $f$ satisfies $(\star)$, then for $p$ sufficiently big, and $\sigma$ sufficiently small (depending on $p$), $f$ is uniformly bounded in spacetime.  The bound will depend on $(S, \Sigma_0, T, n, \alpha, p, \sigma)$.
\end{theorem}

The following Lemma is the key step in handling the free boundary behavior.  We first make a useful observation.

\begin{remark}\label{remark:handle-low-H-powers}
Let $g$ be an arbitrary non-negative function on $\Sigma_t$.  If $r \in (0, 2)$, and $q \in (0, p)$ with $rp/q < 2$, then for any $\mu > 0$, 
\begin{align*}
\int g^q \tH^r &\leq \int g^p \tH^{rp/q} + |\spt g| \\
&\leq \frac{1}{\mu} \int g^p \tH^2 + C(\mu, r, q, p) \int g^p + |\spt g|
\end{align*}
\end{remark}

\begin{lemma}\label{lemma:push-boundary-inside}
For any $\mu > 0$ and $p > 4$, $\sigma < 1/2$, we can pick constants $c = c(n, S)$ and $C = (n, S, \mu, p)$ such that
\begin{align*}
 \int_{\partial\Sigma_t} f^{p-1}_k \tilde H^\sigma
&\leq c \int_{\Sigma_t} |\nabla f|^2 f_k^{p-2} + c\sigma \int_{\Sigma_t} \frac{\tilde G^2}{\tilde H^{2-\sigma}} f_k^{p-1} + \frac{cp^2}{\mu} \int_{A(k, t)} f^p \tilde H^2 \\
&\quad C \int_{A(k, t)} f^p + C|A(k, t)|
\end{align*}
\end{lemma}

\begin{proof}
Using the trace formula of \ref{lemma:trace-formula}, and Peter-Paul, we have (all integrals on the right-hand-side are over $\Sigma_t$)
\begin{align*}
\int_{\partial\Sigma_t} f^{p-1}_k\tH^\sigma
&\leq cp \int f^{p-2}_k |\nabla f| \tH^\sigma + c\sigma \int f^{p-1}_k \tH^{\sigma-1}|\nabla \tH| \\
&\quad + c \int f^{p-1}_k \tH^{1+\sigma} + c \int f^{p-1}_k\tH^\sigma \\
&\leq c\int f^{p-2}_k |\nabla f|^2 + cp^2 \int f_k^{p-2} \tH^{2\sigma} + c\sigma\int f_k^{p-1} \frac{\tilde G^2}{\tH^{2-\sigma}} \\
&\quad + c \int f_k^{p-1} (\tH^{1+\sigma} + \tH^\sigma) .
\end{align*}
The Lemma follows by Remark \ref{remark:handle-low-H-powers}. 
\end{proof}

The hard part of Theorem \ref{theorem:f-is-bounded} is establishing $L^p$ bounds for appropriately large $\sigma$.  In particular, we establish spacetime $L^p$ bounds for $\sigma \sim p^{-1/2}$, rather than the naive $\sigma \sim p^{-1}$, and thereby have the following wiggle room.
\begin{lemma}\label{lemma:wiggle-room}
Suppose there is a $p_0$ and $c_\sigma$, independent of $p, \sigma$, such that whenever $p > p_0$ and $\sigma < \frac{c_\sigma}{\sqrt{p}}$,
\[
\int_0^T \int_{\Sigma_t} f^p < \infty .
\]
Then for $m > 0$, 
\[
\int_0^T \int_{\Sigma_t} \tilde H^m f^p < \infty
\]
provided $p > 4m^2/c_\sigma^2 + p_0$ and $\sigma < \frac{c_\sigma}{2\sqrt{p}}$.
\end{lemma}
\begin{proof}
Follows directly from
\[
\tilde H^m f^p = (f_\alpha \tilde H^{\sigma + m/p})^p . \qedhere
\]
\end{proof}

\begin{lemma}\label{lemma:Lp-bound}
Given $(\star)$, then
\[
\int_0^T \int_{\Sigma_t} f^p < \infty
\]
for $p > p_0(c)$, and $\sigma < c_\sigma(c) p^{-1/2}$.
\end{lemma}

\begin{proof}
Combining equations (POINCARE-LIKE), (EVOLUTION-LIKE), and Lemma \ref{lemma:push-boundary-inside}, we have the following inequalities.  We adhere to the convention $c = c(S, \Sigma_0, T, n, \alpha)$ and $C = C(S, \Sigma_0, T, n, \alpha, p, \sigma, \mu)$.  Unless stated otherwise all integrals are on $\Sigma_t$.
\begin{align*}
\partial_t \int f^p
&\leq -p^2/3 \int |\nabla f|^2 f^{p-2} - p/c \int \frac{\tilde G^2}{\tilde H^{2-\sigma}} f^{p-1} \\
&\quad + cp\sigma \left[ p(1+1/\beta) \int |\nabla f|^2 f^{p-2} + (1+\beta p) \int \frac{\tilde G^2}{\tilde H^{2-\sigma}} f^{p-1} \right. \\
& \quad \quad \quad \left. + \int f^p + \int_{\partial \Sigma_t} f^{p-1}\tilde H^\sigma \right] \\
&\quad -1/5 \int f^p \tilde H^2 + C \int f^p + C |\Sigma_t| + cp \int_{\partial\Sigma_t} f^{p-1}\tilde H^\sigma \\
&\leq (-p^2/3 + cp^2\sigma(1+1/\beta)) \int |\nabla f|^2 f^{p-2} \\
&\quad +  (-p/c + cp\sigma(1+\beta p)) \int \frac{\tilde G^2}{\tilde H^{2-\sigma}} f^{p-1} \\
&\quad + cp\left[ \int |\nabla f|^2 f^{p-2} + \sigma \int \frac{\tilde G^2}{\tilde H^{2-\sigma}} f^{p-1} + p^2/\mu \int f^p \tilde H^2 + C\int f^p + C|\Sigma_t| \right] \\
&\quad + C \int f^p + C |\Sigma_t| - 1/5 \int f^p \tilde H^2 \\
&\leq (-p^2/3 + cp^2\sigma(1+1/\beta) + cp) \int |\nabla f|^2 f^{p-2} \\
&\quad + (-p/c + cp\sigma(1+\beta p) + cp\sigma) \int \frac{\tilde G^2}{\tilde H^{2-\sigma}} f^{p-1}  \\
&\quad + (cp^3/\mu -1/5) \int f^p\tilde H^2 + C\int f^p + C|\Sigma_t|
\end{align*}
Choose $\sigma = \frac{1}{6} (c^3p)^{-1/2}$, $\beta = (cp)^{-1/2}$ and $\mu = 10cp^3$, then for $p > 12c$ we have that $\int_{\Sigma_t} f^p$ increases at most exponentially.
\end{proof}

Now for arbitrary $k$, we can combine equation (EVOLUTION-LIKE) with Lemma \ref{lemma:push-boundary-inside} in an identical manner to obtain
\begin{align*}
\partial_t \int_{\Sigma_t} f_k^p \leq -p^2/12 \int_{\Sigma_t} |\nabla f|^2 f^{p-2}_k + C \int_{A(k, t)} f^p \tilde H^2 + C \int_{A(k, t)} f^p + C|A(k,t)|
\end{align*}
for $\sigma$, and $p$ satisfying the same bounds as Lemma \ref{lemma:Lp-bound}.  Here, as in Lemma \ref{lemma:Lp-bound}, $c$ and $C$ are both independent of $k$.

The following Theorem will complete the proof of Theorem \ref{theorem:f-is-bounded}.
\begin{theorem}\label{theorem:stampacchia}
Suppose there is a $p_0$ and $c_\sigma$, independent of $p, \sigma, k$, such that whenever $p > p_0$ and $\sigma < \frac{c_\sigma}{\sqrt{p}}$, we have
\[
\int_0^T \int_{\Sigma_t} f^p < \infty
\]
and
\begin{equation}\label{eqn:stampacchia-equation}
\partial_t \int_{\Sigma_t} f_k^p + 1/c \int_{\Sigma_t} |\nabla f_k^{p/2}|^2 \leq C \int_{A(k, t)} \tilde H^2 f^p + C \int_{A(k, t)} f^p + C|A(k, t)|
\end{equation}
for any $k > 0$.  Here $c$ and $C$ can depend on any quantity except $k$.  Then for $p$ sufficiently large, and $\sigma$ sufficiently small, $f^p$ is uniformly bounded in spacetime.  The bound will depend on $(S, \Sigma_0, T, n, p, \sigma, \alpha)$.
\end{theorem}

\begin{proof}
By Theorem \ref{theorem:micheal-simon-free-boundary} and Corollary \ref{corollary:micheal-simon-n=2}, for each $n \geq 2$ there is a $q > 1$, and $c = c(n, q, |\Sigma_0|)$, such that
\[
\left( \int_\Sigma v^{2q}\right)^{1/q} \leq c\int_\Sigma |Dv|^2 + c\int_\Sigma v^2 H^2 + c \int_\Sigma v^2 .
\]
So take $v = f^{p/2}_k$ and integrate \eqref{eqn:stampacchia-equation} to obtain (for possibly larger $C$)
\[
\max\left\{ \sup_{[0,T)} \int_{\Sigma_t} f_k^p, \int_0^T \left( \int_{\Sigma_t} f_k^{pq} \right)^{1/q} \right\}
\leq C\iint_{A(k)} f^p + C \iint_{A(k)} \tilde H^2 f^p + C |A(k)| .
\]
provided $k \geq k_0(\Sigma_0, n, p, \sigma, \alpha)$.  All terms on the right are bounded by virtue of Lemma \ref{lemma:wiggle-room}, and the monotonicity of $|\Sigma_t|$.  Therefore
\begin{align*}
\int_0^T \int f_k^{p\frac{2q-1}{q}} 
&\leq \int_0^T \left( \int f_k^{pq} \right)^{1/q} \left( \int f_k^p \right)^{\frac{q - 1}{q}} \\
&\leq C \left( \iint_{A(k)} f^p + \iint_{A(k)} \tilde H^2 f^p + |A(k)|\right)^{\frac{2q-1}{q}} \\
&\leq C |A(k)|^{\frac{2q - 1}{q} (1 - 1/r)} \left( \left(\iint_{A(k)} f^{pr}\right)^{1/r} + \left(\iint_{A(k)} \tilde H^{2r} f^{pr}\right)^{1/r} + |A(k)|^{1/r} \right)^{\frac{2q - 1}{q}} \\
&\leq C(S, \alpha, p, \sigma, T, c_\sigma, \Sigma_0) |A(k)|^\alpha
\end{align*}
for any $r$, provided $p > 16r/c_\sigma^2 + p_0$ and $\sigma < \frac{c_\sigma}{2\sqrt{p}}$.  If we fix $r$ sufficiently large, then $\alpha = \frac{2q - 1}{q} (1 - 1/r) > 1$.  Fix $p$, $\sigma$, then for any $\ell > k$, we have the inequality
\begin{equation} \label{eqn:proto-stampacchia}
|\ell - k|^\beta |A(\ell)| \leq C |A(k)|^\alpha
\end{equation}
where $\beta = p\frac{2q - 1}{q} > 0$, and $C$ is independent of $\ell, k$.  It follows by a standard argument that $A(k) = 0$ for $k > k_0(\alpha, \beta, C)$, $C$ as in \eqref{eqn:proto-stampacchia}.
\end{proof}

\section{Mean curvature flow with free boundary preliminaries}
\label{section:preliminaries}

Let $(\Sigma_t)_{t \in [0, T)}$ be the mean curvature flow of $\Sigma_0$, with free-boundary in $S$.  Here, as always in this paper, $T$ is the maximal time of existence.

Write $g = (g_{ij})$ and $A = (h_{ij})$ for the induced metric and second fundamental form on $\Sigma_t$.  We follow the usual convention that $g^{ij}$ is the matrix inverse to $g_{ij}$, and a raised index such as $h^i_j$ means $\sum_k g^{ik}h_{kj}$.  We denote $dV$ the volume form on $\Sigma_t$, and take $N$ for the outward normal of $\partial\Sigma \subset \Sigma$.

We write $\nabla$ for covariant differentiation in $\Sigma$, and $\bar\nabla$ for covariant differentiation in $R^{n+1}$.  We write $(k_{ij})$ for the second fundamental form of the barrier surface $S$.

\begin{proposition}\label{prop:general-evolution-equations}
We have the following evolution equations, using summation convention.
\begin{align*}
&\partial_t g_{ij} = -2H h_{ij} \\
&\partial_t h_{ij} = \Delta h_{ij} - 2Hh_{im} h^m_j + |A|^2 h_{ij}
\end{align*}
and
\begin{align*}
&\partial_t H = \Delta H + |A|^2 H \\
&\partial_t dV = -H^2 dV \\
&\partial_t \nu = \nabla H
\end{align*}
\end{proposition}

\begin{proof}
See \cite{huisken:umbilic-pinching}.
\end{proof}

\begin{remark}\label{remark:area-decreasing}
Since the boundary $\partial\Sigma_t$ is always orthogonal to the direction of motion, 
\[
\partial_t |\Sigma_t| = - \int_{\Sigma_t} H^2 \leq 0 .
\]
Specifying other angles of contact would add a boundary term to $\partial_t |\Sigma_t|$, and could even cause area increase.
\end{remark}

\begin{proposition}\label{prop:H-boundary-derivative}
We have
\[
N(H) = k_{\nu\nu} H .
\]
In particular, positivity of $H$ is preserved for all time.  If $S$ is convex, then $H$ is non-decreasing, and in fact must blow up in finite time.
\end{proposition}

\begin{proof}
Differentiate the relation $<N, \nu> = 0$ in time.  Evolution behavior follows from Proposition \ref{prop:general-evolution-equations}.
\end{proof}

\begin{remark}
Notice that $H$ may still decrease.  We will show later that $H$ decreases at worst exponentially in time.
\end{remark}

\begin{proposition}
For any $X \in T_p\partial\Sigma$, 
\[
h_{N, X} = -k_{\nu, X} .
\]
\end{proposition}

\begin{proof}
Since $N \equiv \nu_S$ along $T_p\partial\Sigma$, 
\[
h(N, X) = -<\nu, \bar \nabla_X \nu_S> = -k(\nu, X) \qedhere
\]
\end{proof}

As mentioned in the Introduction the key technical issue in extending the estimates to general barrier surfaces is in calculating $\nabla_N h_{N,X} = \nabla_X h_{N,N}$, for $X \in T_p\partial\Sigma$.  To avoid the issue we perturb $h$ so that $h_{N, X} = 0$.

\begin{definition}
Extend and fix $k$ and $\nu_S$ to be defined on $R^{n+1}$.  Define the \emph{perturbed second fundamental form} $\bA$ of $\Sigma$ to be
\begin{equation}
\bar h_{ij} = h_{ij} + T_{ij\nu} + D_0 g_{ij}
\end{equation}
where $T$ is a 3-tensor defined on the ambient space by
\[
T(X, Y, Z) = k(X, Z) g(Y, \nu_S) + k(Y, Z) g(X, \nu_S) .
\]
We choose and fix the constant $D_0$ so that
\[
T(X, X, \nu) + D_0 \geq 1
\]
for any unit vector $X$.  From henceforth when a constant depends on $D_0$ or the extensions of $k$ or $\nu_S$, we will only say it depends on the barrier surface $S$.
\end{definition}

Our choice of $D_0$ and Proposition \ref{prop:H-boundary-derivative} imply that
\begin{equation}
\bH \geq H + 1 \geq 1, \quad |\bA| \geq 1 .
\end{equation}

\section{Evolution of tensors}
\label{section:evolutions}

\begin{proposition}\label{prop:tensor-evolution}
Let $T$ be a 3-tensor defined on the ambient space.  If $T_{ij\nu}$ is the 2-tensor $T(\cdot, \cdot, \nu)$ restricted to $T\Sigma$, then
\begin{align*}
&\nabla T_{ij\nu} = O(1 + |A|) \\
&\nabla^2 T_{ij\nu} = O(1 + |A|^2 + |\nabla A|) \\
&(\partial_t - \Delta) T_{ij\nu} = O(1 + |A|^2) \\
&(\partial_t - \Delta) T^i_{j\nu} = O(1 + |A|^2) .
\end{align*}
\end{proposition}

\begin{proof}
Choose orthonormal geodesic coordinates $\partial_i$ at a fixed point $p$.  We use the summation convention, excepting of course on $\nu$.  We have
\begin{align*}
\nabla_p T_{ij\nu} 
&= \bar\nabla_p T_{ij\nu} + T_{ij \bar\nabla_p \nu} + T_{\nabla^\perp_p i j \nu} + T_{i \nabla^\perp_p j \nu} \\
&= \bar\nabla_p T_{ij\nu} + h_{pk} T_{ijk} - h_{pi} T_{\nu j \nu} - h_{pj} T_{i \nu \nu} \\
&= O(1 + |A|)
\end{align*}

We work towards calculating $\nabla^2 T$ and $\Delta T$.  We have
\begin{align*}
\nabla_q (h_{pi} T_{\nu j\nu})
&= (\nabla_q h_{pi}) T_{\nu j\nu} + h_{pi} \nabla_q T_{\nu j\nu} \\
&= \nabla_i h_{pq} T_{\nu j\nu} + h_{pi} (\bar\nabla_q T_{\nu j\nu} + h_{qk} T_{\nu j k} + h_{qk} T_{k j \nu} - h_{qj} T_{\nu \nu \nu} ) \\
&= \nabla_i h_{pq} T_{\nu j \nu} + O(1 + |A|^2)
\end{align*}
and
\begin{align*}
\nabla_q (h_{pk} T_{ijk})
&= \nabla_k h_{pq} T_{ijk} + h_{pk} (\bar\nabla_q T_{ijk} - h_{qi} T_{\nu jk} - h_{qj} T_{i\nu k} - h_{qk} T_{ij \nu}) \\
&= \nabla_k h_{pq} T_{ijk} + O(1 + |A|^2)
\end{align*}
and
\begin{align*}
\nabla_q \bar\nabla_p T_{ij\nu}
&= \bar\nabla^2_{q,p} T_{ij\nu} + \bar \nabla_{\nabla^\perp_q p} T_{ij\nu} + \bar\nabla_p T_{\nabla^\perp_q i j \nu} + \bar\nabla_p T_{i \nabla^\perp_q j \nu} + \bar\nabla_p T_{i j \nabla_q \nu} \\
&= \bar\nabla^2_{q,p} T_{ij\nu} - h_{qp} \bar\nabla_\nu T_{ij\nu} - h_{qi} \bar\nabla_p T_{\nu j\nu} - h_{qj} \bar\nabla_p T_{i\nu\nu} + h_{qk} \bar\nabla_p T_{ijk} \\
&= O(1 + |A|) .
\end{align*}
We therefore have
\begin{align*}
\nabla^2_{q,p} T_{ij\nu} &= \nabla_q (\bar\nabla_p T_{ij\nu} + h_{pk} T_{ijk} - h_{pi} T_{\nu j\nu} - h_{pj} T_{i\nu\nu} )\\
&= \nabla_k h_{pq} T_{ijk} - \nabla_i h_{pq} T_{\nu j \nu} - \nabla_j h_{pq} T_{i \nu \nu} + O(1 + |A|^2) \\
&= O(1 + |A|^2 + |\nabla A|)
\end{align*}
and
\begin{align*}
\Delta T_{ij\nu}
&= \partial_k H T_{ijk} - \partial_i H T_{\nu j \nu} - \partial_j H T_{i \nu \nu} + O(1 + |A|^2) \\
&= O(1 + |A|^2 + |\nabla H|) .
\end{align*}

We calculate the time derivative.  Here $(\bar x^\alpha)$ are standard coordinates in $R^{n+1}$.
\begin{align*}
\partial_t T_{ij\nu}
&= \partial_t \left(T_{\alpha\beta\gamma}(F(x)) \frac{\partial F^\alpha}{\partial x^i} \frac{\partial F^\beta}{\partial x^j} \nu^\gamma\right) \\
&= \left(\frac{\partial T_{\alpha\beta\gamma}}{\bar x^\delta} \partial_t F^\delta\right) \partial_i^\alpha F \partial_j^\beta F \nu^\gamma + T_{\alpha\beta\gamma} \left( \frac{\partial}{\partial x^i} \frac{\partial F^\alpha}{\partial t} \right) \partial_j^\beta F \nu^\gamma \\
&\quad + T_{\alpha\beta\gamma} \partial_i^\alpha F \left(\frac{\partial}{\partial x^j} \frac{\partial F^\beta}{\partial t} \right) \nu^\gamma + T_{\alpha\beta\gamma} \partial_i^\alpha F \partial_j^\beta F \frac{\nu^\gamma}{\partial t} \\
&= -H \bar\nabla_\nu T_{ij\nu} + T( \bar\nabla_i (-H \nu), \partial_j, \nu) + T(\partial_i, \bar\nabla_j (-H\nu), \nu) + T(\partial_i, \partial_j, \nabla H) \\
&= -H \bar\nabla_\nu T_{ij\nu} - \partial_i H T_{\nu j \nu} - H h_{ik} T_{kj\nu} - \partial_j H T_{i \nu \nu} - H h_{jk} T_{i k \nu} + \partial_k H T_{ijk} \\
&= -\partial_i H T_{\nu j \nu} - \partial_j H T_{i \nu \nu} + \partial_k H T_{ijk} + O(1 + |A|^2)
\end{align*}
which proves the penultimate formula.  The last formula follows by observing that $\partial_t g_{ij} = O(|A|^2)$.
\end{proof}

\begin{corollary}\label{corollary:orders}
We have
\begin{align*}
1 = O(|\bA|), \quad |A| = O(|\bA|), \quad |\nabla A| = O(|\nabla \bA| + |\bA|) .
\end{align*}
\end{corollary}

\begin{proof}
The first formula follows trivially from $|\bA| \geq 1$.  The second because $\bar A = A + O(1)$.  The third since $\nabla A = \nabla A + O(|A| + 1)$.
\end{proof}

\begin{theorem}\label{thm:A-H-evolution}
We have the evolution equations
\begin{align*}
&\partial_t \bar h^i_j = \Delta \bar h^i_j + |\bA|^2 \bar h^i_j + O(|\bA|^2) \\
&\partial_t |\bA|^2 = \Delta |\bA|^2 + 2|\bA|^4 - 2|\nabla \bA|^2 + O(|\bA|^3) \\
&\partial_t H = \Delta H + |\bA|^2 H + O(|\bA|)H
\end{align*}
\end{theorem}

\begin{proof}

We deduce the first formula by Propositions \ref{prop:general-evolution-equations} and \ref{prop:tensor-evolution}.

We have
\begin{align*}
\frac{1}{2} (\partial_t - \Delta) |\bA|^2 
&= \frac{1}{2} \partial_t ( g^{ik} g^{kl} \bh_{ij} \bh_{kl} ) - <\Delta \bA, \bA> - |\nabla \bA|^2 \\
&= 2H h^{ik} g^{jl} \bh_{ij} \bh_{kl} + g^{ik}g^{jl} (\partial_t - \Delta) (h_{ij} + T_{ij\nu} + D g_{ij}) \bh_{kl} - |\nabla \bA|^2 \\
&= 2H h^{ik} g^{jl} (\bh_{ij} \bh_{kl} - h_{ij} \bh_{kl}) + |A|^2 <A, \bA> + O(1 + |\bA|^3) \\
&\quad - 2 D_0Hg^{ik} g^{jl} h_{ij} \bh_{kl} - |\nabla \bA|^2\\
&= |\bA|^4 - |\nabla \bA|^2 + O(|\bA|^3) .
\end{align*}

The third formula is an immediate consequence of Proposition \ref{prop:general-evolution-equations} and Corollary \ref{corollary:orders}.
\end{proof}

\begin{lemma}\label{lemma:matrix-inequality}
Let $M$ be a symmetric matrix, and $\eta > 0$.  If $|M| > (1+\eta)\tr(M)$, then
\[
|M|^2 - \max_i |\lambda_i|^2 \geq \frac{1}{c} |M|^2 .
\]
Here $\{ \lambda_i\}$ are the eigenvalues of $M$, and $c = c(n, \eta)$.
\end{lemma}

\begin{proof}
Otherwise, we can pick a sequence of counterexamples $M^{(j)}$ with $|M^{(j)}| = 1$ and
\[
|M^{(j)}|^2 - \max_i |\lambda_i^{(j)}|^2 \leq |M^{(j)}|^2/j = 1/j .
\]
Since each entry lies in the interval $[-1, 1]$, we can pick a subsequence $M^{(j')}$ converging to $M$.  Then all but one eigenvalue of $M$ is zero, contradicting $|M| \geq (1+\eta)\tr(M)$.
\end{proof}

\begin{proposition}\label{prop:A-H-gradient-term}
If $|\bA| > 2\bH$, then
\begin{equation}
|\nabla \bA|^2 - |\nabla|\bA||^2 \geq \frac{1}{c} |\nabla \bA|^2 + O(|\bA|^2)
\end{equation}
where $c = c(n)$.
\end{proposition}

\begin{proof}
We have that
\[
\nabla_i \bh_{jk} = \nabla_j \bh_{ik} + O(|\bA|)
\]
and therefore, if we pick orthonormal coordinates so that $\partial_1 = \nabla |\bA| / |\nabla |\bA||$ at the point in question,
\begin{align*}
|\bA|^2(|\nabla \bA|^2 - |\nabla|\bA||^2)
&= \left||\bA| \nabla_i \bh_{jk} - \nabla_i |\bA| \bh_{jk}\right|^2 \\
&= \left||\bA| \nabla_i \bh_{jk} - \frac{1}{2}(\nabla_i |\bA| \bh_{jk} + \nabla_j |\bA| \bh_{ik} ) + \frac{1}{2} (\nabla_j |\bA| \bh_{ik} - \nabla_i |\bA| \bh_{jk}) \right|^2 \\
&\geq \frac{1}{4}\left|\nabla_j |\bA| \bh_{ik} - \nabla_i |\bA| \bh_{jk}\right|^2 - c |\bA|^3 |\nabla |\bA|| \\
&\geq \frac{1}{2} |\nabla|\bA||^2 (|\bA|^2 - \sum_{k} \bh_{1k}^2) - c |\bA|^3 |\nabla |\bA|| \\
&\geq \frac{1}{2} |\nabla |\bA||^2(|\bA|^2 - \max_i |\bar\lambda_i|^2) - c |\bA|^3|\nabla|\bA|| .
\end{align*}
Here $c = c(n, S)$, and $\bar\lambda_i$ are the eigenvalues of $\bA$.

By Lemma \ref{lemma:matrix-inequality} there is a $c_n$ depending only on $n$ so that
\[
|\bA|^2 - \max_i |\bar\lambda_i|^2 \geq \frac{1}{c_n} |\bA|^2
\]
and hence by Peter-Paul we deduce that
\[
|\nabla \bA|^2 - |\nabla|\bA||^2 > \frac{1}{2c_{n}}|\nabla|\bA||^2 - c |A|^2 .
\]
This can be rearranged to deduce
\[
|\nabla\bA|^2 - |\nabla|\bA||^2 > \frac{1}{2c_n+1}|\nabla\bA|^2 - c |A|^2 \qedhere
\]
\end{proof}

\begin{corollary}\label{cor:A-evolution}
Whenever $|\bA| > 2\bH$,
\begin{equation}
(\partial_t - \Delta)|\bA| \leq |\bA|^3 - \frac{1}{c} \frac{|\nabla \bA|^2}{|\bA|} + O(|\bA|^2)
\end{equation}
where $c = c(n)$.
\end{corollary}

\begin{proof}
We have (recalling $|\bA| \geq 1$)
\begin{align*}
(\partial_t - \Delta) |\bA|
&= (\partial_t - \Delta) \sqrt{|\bA|^2} \\
&= \frac{1}{2} \frac{(\partial_t - \Delta)|\bA|}{|\bA|} + \frac{1}{4} \frac{|\nabla|\bA|^2|^2}{|\bA|^3} \\
&= |\bA|^3 + \frac{|\nabla|\bA||^2 - |\nabla \bA|^2}{|\bA|} + O(|\bA|^2) .
\end{align*}
Now apply Proposition \ref{prop:A-H-gradient-term}.
\end{proof}


\section{Boundary derivatives}
\label{section:boundary}

Fix a $p \in \partial\Sigma$.  Choose coordinates so that $\partial_1 \equiv N$ along $\partial\Sigma$, $(\partial_i)_{i > 1}$ are orthonormal geodesic normal coordinates on $\partial\Sigma$ at $p$, and the integral curves of $\partial_1$ are geodesics.

\begin{lemma} \label{lemma:boundary-derivatives}
At $p$ we have, for $i, j > 1$,
\begin{align*}
&\nabla_1 h_{ij} = h_{ij} k_{\nu\nu} + h_{11} k_{ij} - k_{i\alpha} h_{j\alpha} - k_{j\alpha}h_{i\alpha} - \nabla^S_\nu k_{ij} \\
&\nabla_1 h_{11} = 2(k_{\alpha\beta} h_{\alpha\beta} + h_{11}k_{\nu\nu}) - Kh_{11} + \nu(K) - \nabla^S_\nu k_{\nu\nu}
\end{align*}
where $\alpha, \beta$ are summed over $2, \ldots, n$, and $K$ is the mean curvature of the barrier $S$.
\end{lemma}

\begin{proof}
We calculate for $i, j > 1$
\begin{align*}
\partial_1 h_{ij}
&= -<\partial_i\partial_j N, \nu> - <\partial_i\partial_j F, \partial_1 \nu> \\
&= -<\partial_i(k_j^\alpha\partial_\alpha + k_j^\nu \nu), \nu> + h_{11}k_{ij} \\
&= k_{j\alpha} h_{i\alpha} - \nabla^S_i k_j^\nu - k_{\nabla^S_i \partial j}^\nu - k_j^{\nabla^S_i \nu} + h_{11}k_{ij} \\
&= -\nabla^S_\nu k_{ij} + h_{ij}k_{\nu\nu} + h_{11} k_{ij}
\end{align*}
and hence
\begin{align*}
\nabla_1 h_{ij}
&= \partial_1 h_{ij} - h((\partial_i \partial_1 F)^T, \partial_j) - h(\partial_i, (\partial_j \partial_1)^T) \\
&= -k_{i\alpha}h_{j\alpha} - k_{j\alpha}h_{i\alpha} - \nabla^S_\nu k_{ij} + h_{ij}k_{\nu\nu} + h_{11}k_{ij}
\end{align*}

We calculate, using Proposition \ref{prop:H-boundary-derivative},
\begin{align*}
N(H) &= k_{\nu \nu} H \\
&= \nabla_1 h_{11} + tr_{\partial \Sigma}(\nabla_1 h) \\
&= \nabla_1 h_{11} - 2k_{\alpha\beta}h_{\alpha\beta} - tr_{\partial\Sigma}(\nabla^S_\nu k_{ij}) + (H - h_{11})k_{\nu\nu} + h_{11}(K - k_{\nu\nu}) \\
&= \nabla_1 h_{11} - 2k_{\alpha\beta}h_{\alpha\beta} - 2h_{11}k_{\nu\nu} - \nu(K) + \nabla_\nu^S k_{\nu\nu} + Hk_{\nu\nu} + Kh_{11}
\end{align*}
and the Lemma follows.
\end{proof}

\begin{theorem}\label{theorem:boundary-derivatives}
At $p$, for $i, j > 1$,
\begin{align*}
&\nabla_1 \bh_{ij} = O(|\bA|) \\
&\nabla_1 \bh_{11} = O(|\bA|) .
\end{align*}
\end{theorem}

\begin{proof}
Follows directly from Lemma \ref{lemma:boundary-derivatives} using Theorem \ref{prop:tensor-evolution}.
\end{proof}

\begin{theorem}
We have that
\begin{align*}
&N |\bA| = O(|\bA|) .
\end{align*}
\end{theorem}

\begin{proof}
Immediate from Theorem \ref{theorem:boundary-derivatives} and that $\bar h_{N, X} = 0$ when $X \in T_p \partial\Sigma$.
\end{proof}


\section{Controlling $|\bA|$}
\label{section:controlling-A}

In this section we prove the following Theorem, which will imply Theorem \ref{theorem:A-H-bound-intro}.

\begin{theorem}\label{theorem:A-H-bound}
There are constants $\alpha = \alpha(S, n) \geq 0$ and $C = C(S, \Sigma_0)$ so that
\begin{equation}\label{eqn:A-H-bound}
\max_{\Sigma_t} \frac{|\bA|}{H} \leq C e^{\alpha t}
\end{equation}
for all time of existence.
\end{theorem}

\begin{remark}
If $S$ is convex, then $H$ is non-decreasing, and by carefully calculating the normal derivative $N|\bA|$ one can take $\alpha = 0$ in \eqref{eqn:A-H-bound}.
\end{remark}

For arbitrary function $f$ and $g$, recall the useful formula
\begin{equation}\label{eqn:useful-formula}
(\partial_t - \Delta) \frac{f}{g} = \frac{(\partial_t - \Delta) f}{g} - \frac{f}{g^2} (\partial_t - \Delta) g + \frac{2}{g} <\nabla \frac{f}{g}, \nabla g> .
\end{equation}

\begin{proof}[Proof of Theorem \ref{theorem:A-H-bound}]

Recall that
\begin{align*}
&|N H| \leq b H \\
&|N |\bA|| \leq b |\bA| 
\end{align*}
for some constant $b = b(S)$.

Let $d : R^{n+1} \to [-1,1]$ be a smooth function such that $d \equiv 0$ on $S$, and $\nu_S(d) \geq 1$.  If a constant depends on $d$ we will only say it depends on $S$.  Let $\phi : R^{n+1} \to R_+$ be the smooth function
\[
\phi(x) = e^{-\alpha t -2 b d}
\]
so that $\nu_S(\phi) \leq -2b \phi$.

We have, in geodesic orthonormal coordinates, 
\begin{align}
(\partial_t - \Delta) \phi
&= -\alpha \phi + \bar\nabla \phi \cdot (\partial_t F - \Delta F) - \sum_i \bar\nabla^2\phi(\partial_i F, \partial_i F) \nonumber \\
&= -\alpha \phi - \tr_{T\Sigma} (\bar\nabla^2 \phi) \nonumber \\
&= (-\alpha + O(1)) \phi . \nonumber
\end{align}
Choose $\alpha = \alpha(S, n)$ so that $(\partial_t - \Delta)\phi < 0$.

We first show the quantity $\min_{\Sigma_t} H/\phi$ is non-decreasing.  First calculate
\begin{equation}\label{eqn:H-phi-boundary}
N \frac{H}{\phi} \geq b \frac{H}{\phi}, 
\end{equation}
so any spatial minimum is interior.  And by our choice of $\alpha$ we obtain
\begin{equation}\label{eqn:H-phi-evolution}
(\partial_t - \Delta) \frac{H}{\phi} \geq |A|^2 \frac{H}{\phi} + \frac{2}{\phi} <\nabla \frac{H}{\phi}, \nabla \phi> .
\end{equation}

In particular, at any spatial minimum $p$ of $H/\phi$, we must have
\[
\partial_t \frac{H}{\phi}|_p \geq |A|^2 \frac{H}{\phi} \geq 0 .
\]

We now consider the quantity
\[
f = \frac{|\bA| + a}{H/\phi}
\]
for some positive constant $a$ to be determined.  We show $\max_{\Sigma_t} f$ is non-increasing when $f$ is sufficiently big.  At the boundary we have by \eqref{eqn:H-phi-boundary}
\begin{equation}\label{eqn:A-H-boundary}
N f \leq \frac{b |\bA|}{H/\phi} - b H/\phi \frac{|\bA|}{(H/\phi)^2} \leq 0 .
\end{equation}
So any spatial maximum of $f$ is interior.

From Corollary \ref{cor:A-evolution} and equation \eqref{eqn:H-phi-evolution}, whereever $|\bA| > 2\bH$ we have the evolution equations
\begin{align*}
&(\partial_t - \Delta) |\bA| \leq |\bA|^3 - \frac{1}{c_n} \frac{|\nabla |\bA||^2}{|\bA|} + c |\bA|^2 \\
&(\partial_t - \Delta) \frac{H}{\phi} \geq |\bA|^2 \frac{H}{\phi} - c |\bA| \frac{H}{\phi} + \frac{2}{\phi} <\nabla \frac{H}{\phi}, \nabla \phi> .
\end{align*}
Here $c = c(S, n)$ and $c_n = c_n(n)$.

We calculate
\begin{align*}
(\partial_t - \Delta) f 
&\leq \frac{1}{H/\phi} \left( |\bA|^3 - \frac{1}{c_n} \frac{|\nabla |\bA||^2}{|\bA|} + c|\bA|^2 \right) - f (|\bA|^2 - c|\bA|) \\
&\quad - \frac{2f}{H} <\nabla \frac{H}{\phi}, \nabla \phi> + \frac{2}{H/\phi} <\nabla f, \nabla \frac{H}{\phi} > \\
&\leq \frac{\phi}{H} \left\{ |\bA|^3 - \frac{1}{c_n} \frac{|\nabla |\bA||^2}{|\bA|} + c|\bA|^2 - (|\bA| + a)(|\bA|^2 - c|\bA|) \right\} \\
&\quad + \frac{2}{H} |\nabla |\bA|| |\nabla \phi| + <\nabla f, \frac{2}{\phi} \nabla \phi + \frac{2\phi}{H} \nabla \frac{H}{\phi}> \\
&\leq \frac{\phi}{H} \left\{  (2c - a) |\bA|^2 + ac |\bA| - \frac{1}{2c_n} \frac{|\nabla |\bA||^2}{|\bA|} + 2c_n \frac{|\nabla \phi|^2}{\phi^2}|\bA| \right\} \\
&\quad + <\nabla f, \frac{2}{\phi} \nabla \phi + \frac{2\phi}{H} \nabla \frac{H}{\phi}> .
\end{align*}

Notice that $\frac{|\nabla \phi|^2}{\phi^2} = O(1)$.  By the above calculations and equation \eqref{eqn:A-H-boundary}, if $f$ attains its spatial maximum at a point $p$, and $|\bA| > 2\bH$ at this point, then
\begin{align*}
\partial_t f|_p \leq \frac{\phi}{H}\left\{ (c - a)|\bA|^2 + ca|\bA| \right\} \leq 0
\end{align*}
provided we choose $a = 2c$ and ensure $|\bA| > 2c$.

We still need to prove this implies Theorem \ref{theorem:A-H-bound}.  Recall that $\bar H = H + O(1)$.  Using that $\min_{\Sigma_t} H/\phi$ is non-decreasing, we have
\[
\bar H \leq H + c \leq c \frac{H}{\phi} \left( 1 + \frac{1}{\min_{\Sigma_0} H/\phi} \right) .
\]

Define the constant
\[
C = \frac{4c}{\min_{\Sigma_0} H/\phi} + 2c\left( 1 + \frac{1}{\min_{\Sigma_0} H/\phi} \right).
\]
Then if $f \geq C$, we have
\begin{align*}
|\bA| 
&\geq C \frac{H}{\phi} - 2c \\
&\geq 2c ( 1 + (\min H/\phi)^{-1} ) \frac{H}{\phi} + 4c - 2c \\
&\geq 2 \bar H + 2c.
\end{align*}

We deduce that
\[
\frac{|\bA|}{H} \leq \frac{f}{\phi} \leq \phi^{-1} \max\{ C, \max_{\Sigma_0} f\},
\]
which proves the Theorem.
\end{proof}

\begin{proposition}\label{prop:finite-dist}
There are constants $\alpha = \alpha(S, n)$, $C = C(S, \Sigma_0)$ so that
\[
\max_{x \in \Sigma_t} \dist(x, 0) \leq C e^{\alpha t} 
\]
for all time of existence.  In particular, if $T < \infty$, then we can find a radius $R$ satisfying
\[
\bigcup_{t \in [0, T)} \Sigma_t \subset B_R(0) .
\]
\end{proposition}

\begin{proof}
We consider the quantity
\[
f = \phi |F|^2 = \phi \sum_\beta F_\beta^2 ,
\]
where $\phi$ is the cutoff function from Theorem \ref{theorem:A-H-bound}.

At the boundary
\[
N |F|^2 = 2\sum_\beta F_\beta N_\beta \leq 2\max\{ |F|^2, 1 \}.
\]
And in the interior
\[
(\partial_t - \Delta) |F|^2 = 2\sum_\beta F_\beta (\partial_t - \Delta) F_\beta - 2\sum_\beta |\nabla F_\beta|^2 \leq 0 .
\]

Whenever $|F| \geq 1$, we obtain (provided $b \geq 1$)
\[
N f \leq 2 |F|^2 \phi - 2b \phi |F|^2 \leq 0 ,
\]
and hence any spatial maximum on $\{ |F| \geq 1\}$ is interior.

We calculate
\begin{align*}
(\partial_t - \Delta) f &\leq -\frac{2}{\phi} <\nabla f, \nabla \phi> + 2 |F|^2 \phi \frac{|\nabla \phi|^2}{\phi^2} \\
&\leq -\frac{2}{\phi} <\nabla f, \nabla \phi> + c f
\end{align*}
for $c = c(S, n)$.  We deduce that
\[
|F|^2 \leq \phi^{-1} e^{ct} \max\{1, \max_{\Sigma_0} f \}. \qedhere
\]
\end{proof}


\section{Convexity pinching}
\label{section:convexity-pinching}

We prove Theorem \ref{theorem:convexity-pinching}.  Recall that we wish to show that if $T < \infty$, then for any $k \in \{1, \ldots, n\}$ and any $\eta > 0$,
\begin{equation}\label{eqn:convexity-estimate-reminder}
S_k \geq -\eta H^k - C
\end{equation}
with $C = C(S, \Sigma_0, T, \eta, n)$.  Here $S_k$ is the $k$-th symmetric polynomial of the principle curvatures.   We following \cite{huisken-sinestrari:convex-pinching} and prove \eqref{eqn:convexity-estimate-reminder} by induction on $k$.  Notice this is trivially true for $k = 1$.

From henceforth assume \eqref{eqn:convexity-estimate-reminder} holds up to a fixed $k$, i.e. $S_l \geq -\eta H^l - C$ for every $l = 1, \ldots, k$.  We will now prove \eqref{eqn:convexity-estimate-reminder} for $k+1$.  Of course we also from now on assume $T < \infty$.

In spirit we would like to consider the function
\[
\frac{-S_{k+1}/S_k - \eta H}{H} H^\sigma
\]
and show this is bounded above in spacetime.  However for general $k$ we have no positivity control over the denominator $S_k$.  We require a further perturbation of the second fundamental form.

\begin{definition}
Let $\tilde A = (b_{ij})$ be the \emph{twice-perturbed second fundamental form}
\begin{align*}
b_{ij} &= \bar h_{ij} + (\epsilon H + D - D_0)g_{ij} \\
&= h_{ij} + T_{ij\nu} + (\epsilon H + D)g_{ij}.
\end{align*}
Here $D \geq D_0 + 1$ and $\epsilon \in (0, \frac{1}{2n}]$ are constants to be fixed later.
\end{definition}

We write $\tlambda_i$ for the eigenvalues of $b_{ij}$, so that if $\bar\lambda_i$ are the eigenvalues of the first-perturbed $\bar h_{ij}$, then 
\[
\tlambda_i = \bar \lambda_i + (\epsilon H + D - D_0) .
\]
Correspondingly $|\tA|$ is the norm of the twice-perturbed second fundamental form, $\tH$ the mean curvature, and $\tS_k = s_k(\tlambda)$, $\tQ_k = q_k(\tlambda)$ where defined.

Recall we had fixed $D_0 = D_0(S)$ so that $T(X, X, \nu) + D_0 \geq 1$ for any unit vector $X$.  So we still have the conditions
\begin{equation} \label{eqn:tilde-bounded-below}
\tH \geq H + 1 \geq 1, \quad |\tA| \geq 1
\end{equation}
and since $|A| \leq c(S, \Sigma_0, T)H$, we have
\begin{equation} \label{eqn:tilde-bounded-above}
|\tA| \leq c(S, \Sigma_0, T)\tH .
\end{equation}

\begin{remark} \label{remark:orders}
Since $|\tA| \geq 1$ and $\epsilon \leq \frac{1}{2n}$, we have
\[
1 = O(|\tA|), \quad |A| = O(|\tA|), \quad |\nabla A| = O(|\nabla \tA| + |\tA|) .
\]
\end{remark}


\begin{lemma}\label{lemma:scale-breaking-invariance}
If $\bar h_{ij} = h_{ij} + O(1)$, and
\[
S_l \geq -\theta H^l - C
\]
for any $\theta > 0$, then we also have
\[
\bar S_l \geq -\theta \bar H^l - \bar C
\]
for any $\theta > 0$.  Here both $C, \bar C$ depend on $S, \Sigma_0, T, \theta, n$.
\end{lemma}

\begin{proof}
Given $\theta > 0$ and the corresponding $C$, we have for $c = c(S, \Sigma_0, T, n, l)$, 
\begin{align*}
\bar S_l &\geq S_l - cH^{l-1} \\
&\geq -\theta H^l - C - cH^{l-1} \\
&\geq -2\theta \bH^l - C - c\bar H^{l-1} - c \\
&\geq -4\theta \bH^l - C - c - c\left(\frac{c}{\theta}\right)^{l-1} . \qedhere
\end{align*}
\end{proof}

\begin{lemma}\label{lemma:scale-breaking-to-bound}
Suppose for any $l = 1, \ldots, k$ and any $\theta > 0$, we have
\[
S_l \geq -\theta H^l - C.
\]
Then for any $\epsilon \in (0, \frac{1}{2n}]$, there is a $D_\epsilon \geq D_0 + 1$ such that
\begin{equation}\label{eqn:scale-breaking-to-bound}
\tS_k \geq \frac{\epsilon}{1+n\epsilon} \frac{n-k+1}{k} \tS_{k-1} \tH
\end{equation}
whenever $D \geq D_\epsilon$.
\end{lemma}

\begin{proof}
Lemma \ref{lemma:scale-breaking-invariance} implies the hypothesis holds for $\bar S_l$ ($l = 1, \ldots, k$).  Since $b_{ij} = \bh_{ij} + (\epsilon H + D - D_0) g_{ij}$, by Lemma 2.7 of \cite{huisken-sinestrari:convex-pinching} there exists a $D_1 = D_1(\epsilon, S, \Sigma_0, T)$ such that \eqref{eqn:scale-breaking-to-bound} holds whenever $D - D_0 \geq D_1$.  Now set $D_\epsilon = D_1 + D_0 + 1$.
\end{proof}

Although we will fix $\epsilon \in (0, \frac{1}{2n}]$ later, for the duration of the paper we take $D = D_\epsilon$ as in Lemma \ref{lemma:scale-breaking-to-bound}.

\begin{remark}\label{remark:tilde-S-bounded-below}
Our inductive hypothesis and our choice of $D$ implies that, for $l = 1, \ldots, k$
\begin{equation}
\tS_l \geq c(n)\epsilon \tH \tS_{l-1} \geq c(n) \epsilon^{l-1} \tH^l .
\end{equation}
\end{remark}

\begin{remark}[Derivatives of $\tS_l$]\label{remark:derivative-bounds}
$\tS_l$ is a homogeneous degree $l$ polynomial in the entries $b^i_j$.  If $\partial$ denotes differentiation in the entries of $b^i_j$, we have for any $d + s \leq l$ and any $l = 1, \ldots, n$
\[
\left| \partial^d \tS_l \right| \leq c(S, T, \Sigma_0, n) \tH^{l-d}
\]
and
\[
\left| \nabla^s \partial^d \tS_l \right| \leq c(S, T, \Sigma_0, n) \tH^{l-d-s} |\nabla \tA|^s .
\]

Using Remark \ref{remark:tilde-S-bounded-below}, we also get that, for $l = 1, \ldots, k+1$
\[
\left| \partial^d \tQ_l \right| \leq c(S, \Sigma_0, T, n, \epsilon) \tH^{1-d} .
\]
\end{remark}

\begin{definition}
For $\eta, \sigma \in (0, 1]$, let
\[
f = \frac{-\tQ_{k+1} -\eta \tH}{\tH^{1-\sigma}} .
\]
\end{definition}

We see that $f$ is well-defined by Remark \ref{remark:tilde-S-bounded-below} and $f \geq 0$ if and only if
\[
\tQ_{k+1} \leq -\eta \tH .
\]
By Remark \ref{remark:derivative-bounds} we have that
\begin{equation}\label{eqn:f-bounds}
| \partial^d f| \leq c(S, \Sigma_0, T, \epsilon, n) \tH^{\sigma-d} .
\end{equation}

\begin{lemma}\label{lemma:f-bounded-implies}
Suppose for every $\epsilon \in (0, \frac{1}{2n}]$ and $\eta \in (0, 1]$, there exists $\sigma \in (0, 1]$ and $C = C(S, \Sigma_0, T, n, \epsilon, \sigma)$ such that
\[
f_+ < C .
\]
Then for any $\theta > 0$ there is a $\bar C = \bar C(S, \Sigma_0, T,n, \theta)$ such that
\[
S_{k+1} \geq -\theta H^{k+1} - \bar C .
\]
\end{lemma}

\begin{proof}
Recall we have fixed $D = D_\epsilon$.  The proof of Lemma 2.8 in \cite{huisken-sinestrari:convex-pinching} shows the hypotheses imply that
\[
\bar S_{k+1} \geq -\theta \bar H^{k+1} - \bar C
\]
for any $\theta > 0$, and $\bar C = \bar C(S, \Sigma_0, T, \theta, n)$.  Now use Lemma \ref{lemma:scale-breaking-invariance}.
\end{proof}

We work towards bounding $f_+$, for a given $\eta > 0$.  We first calculate the order of boundary derivatives.  Choose orthonormal coordinates at a fixed $p \in \partial\Sigma$ such that $\partial_1 \equiv N$.


\begin{theorem}\label{theorem:b_ij-boundary-derivatives}
At $p$ we have, for $i, j > 1$,
\begin{align*}
\nabla_1 b_{ij} = O(|\tA|) \\
\nabla_1 b_{11} = O(|\tA|)
\end{align*}
\end{theorem}

\begin{proof}
By Theorem \ref{theorem:boundary-derivatives} and Proposition \ref{prop:H-boundary-derivative}, we calcuate
\begin{align*}
\nabla_1 b_{11} 
&= \nabla_1 (\bar h_{11} + (\epsilon H + D) g_{11}) \\
&= O(|A| + 1) + \epsilon \partial_1 H g_{11} \\
&= O(|\bA|)
\end{align*}
and the proof for $i, j > 1$ is identical.
\end{proof}

\begin{corollary}\label{cor:tilde-S-boundary}
For every $l = 1, \ldots, n$, 
\begin{equation}\label{eqn:tilde-S-boundary}
N \tS_l = O(|\tA|^l) .
\end{equation}
\end{corollary}

\begin{proof}
For $l = 1$ this follows from the boundary condition $N H = O(H)$ and Proposition \ref{prop:tensor-evolution}.  For $l > 1$, write $S_l$ as the sum of $l$-by-$l$ minors of $b_i^j$, and use that $b_{N, X} = 0$ for $X \in T_p \partial\Sigma$.
\end{proof}

\begin{theorem} \label{theorem:f-boundary-derivative}
We have
\begin{equation}
|N f| \leq c(S, \Sigma_0, T, n, \epsilon) \tH^\sigma .
\end{equation}
\end{theorem}

\begin{proof}
Immediate from Corollary \ref{cor:tilde-S-boundary} and Remark \ref{remark:tilde-S-bounded-below}.
\end{proof}


We obtain an (EVOLUTION-LIKE) equation for $f$.

\begin{proposition}
\begin{equation}
\partial_t b^i_j = \Delta b^i_j + |\tA|^2 b^i_j + O(D|\tA|^2)
\end{equation}
\end{proposition}

\begin{proof}
By Propositions \ref{prop:general-evolution-equations} and \ref{prop:tensor-evolution}, 
\begin{align*}
(\partial_t - \Delta) b^i_j 
&= |A|^2 (h^i_j + \epsilon H g^i_j) + O(1+|A|^2) \\
&= |A|^2 b^i_j - (D + T_{ij\nu})|A|^2 + O(1 + |A|^2) \\
&= |\tA|^2 b^i_j + O(D|\tA|^2)
\end{align*}
recalling that $D \geq 1$.
\end{proof}

\begin{lemma}\label{lemma:gradient-term}
Let $B > \eta > 0$.  There are constants $c_0 = c_0(n, \eta, B)$ and $c = c(c_0, S, \Sigma_0, T, n, \epsilon)$, such that whenever $-B \tS_k \tS_1 \leq \tS_{k+1} \leq -\eta \tS_1 \tS_k$ we have
\begin{equation}
\frac{\partial^2 \tilde Q_{k+1}}{\partial b_{ij} \partial b_{pq}} \nabla_l b_{ij} \nabla_l b_{pq} \leq -\frac{1}{c_0} \frac{|\nabla \tA|^2}{|\tA|} + c \tH
\end{equation}
\end{lemma}

\begin{proof}
Choose orthonormal coordinates which diagonalize $b_{ij}$.  We have, using the notation of Lemma 2.13 of \cite{huisken-sinestrari:convex-pinching},
\begin{align}
\frac{\partial^2 \tilde Q_{k+1}}{\partial b_{ij} \partial b_{pq}} \nabla_l b_{ij} \nabla_l b_{pq}
&= J(\tilde \lambda, \nabla_l (b_{ij} - T_{ij\nu}), \epsilon) \label{eqn:gradient-term-1} \\
&\quad + 2 \frac{\partial^2 q_{k+1}}{\partial \theta_{ij} \partial \theta_{pq}} (\tilde A) \nabla_l (b_{ij} - T_{ij\nu}) \nabla_l T_{ij\nu} \label{eqn:gradient-term-2}\\
&\quad + \frac{\partial^2 q_{k+1}}{\partial \theta_{ij} \partial \theta_{pq}}(\tilde A) \nabla_l T_{ij\nu} \nabla_l T_{pq\nu}\label{eqn:gradient-term-3}
\end{align}

By this same Lemma 2.13, 
\begin{align*}
J(\tilde\lambda, \nabla_l (b_{ij} - T_{ij\nu}), \epsilon) 
&\leq -\frac{1}{c_0} \frac{|\nabla (\tilde A - T)|^2}{|\tilde A|} \\
&\leq -\frac{1}{2c_0} \frac{|\nabla \tilde A|^2}{|\tilde A|} + \frac{1}{c_0} \frac{|\nabla T|^2}{|\tilde A|}
\end{align*}
for $c_0 = c_0(B, n, \eta)$.

By Theorem 2.5 and Lemma 2.12 of \cite{huisken-sinestrari:convex-pinching}, term \eqref{eqn:gradient-term-3} is non-positive.  We bound term \eqref{eqn:gradient-term-2}.  Recall that $|\nabla T| = O(H+1) = O(\tH)$.  Using Remark \ref{remark:derivative-bounds}
\begin{align*}
2 \frac{\partial^2 q_{k+1}}{\partial \theta_{ij} \partial \theta_{pq}} (\tilde A) \nabla_l (b_{ij} - T_{ij\nu}) \nabla_l T_{ij\nu}
&\leq 2\left| \frac{\partial^2 q_{k+1}}{\partial \theta_{ij} \partial \theta_{pq}} (\tA) \right| (|\nabla \tA| + |\nabla T|) |\nabla T| \\
&\leq c |\nabla \tA| + c \tH
\end{align*}
where $c = c(S, \Sigma_0, T, n, \epsilon)$.

We deduce
\begin{align*}
\frac{\partial^2 \tilde Q_{k+1}}{\partial b_{ij} \partial b_{pq}} \nabla_l b_{ij} \nabla_l b_{pq}
&\leq -\frac{1}{4c_0} \frac{|\nabla \tA|^2}{|\tA|} + c \tH
\end{align*}
for $c = c(S, \Sigma_0, T, n, \epsilon, c_0)$.
\end{proof}


Since $f$ is a homogeneous, degree $\sigma$, symmetric function of the eigenvalues $\tilde\lambda_i$ of $b^i_j$, we obtain that
\begin{align*}
(\partial_t - \Delta)f &= \frac{\partial f}{\partial b^i_j} (|\tA|^2 b^i_j + O(|\tA|^2D)) - \frac{\partial^2 f}{\partial b^i_j \partial b^p_q} \nabla_l b^i_j \nabla^l b^p_q \\
&\leq -\frac{\partial^2 f}{\partial b^i_j \partial b^p_q} \nabla_l b^i_j \nabla^l b^p_q + \sigma|\tA|^2 f + c D|\tA|^2 \sum \left|\frac{\partial f}{\partial b^i_j}\right| \\
&\leq -\frac{\partial^2 f}{\partial b^i_j \partial b^p_q} \nabla_l b^i_j \nabla^l b^p_q + \sigma|\tA|^2 f + c D \tH^{1+\sigma}
\end{align*}
for $c = c(S, \Sigma_0, T, n, \epsilon)$.  In the last line we used the inequality \eqref{eqn:f-bounds}.

Lemma \ref{lemma:gradient-term} allows us to crucially obtain a gradient term wherever $f$ is non-negative: on $\spt f_+$ we have
\begin{align}
(\partial_t - \Delta) f 
&\leq \frac{2(1-\sigma)}{\tH} <\nabla \tH, \nabla f> - \frac{\sigma(1-\sigma)}{\tH^2} f|\nabla \tH|^2 \nonumber \\
&\quad + \frac{1}{\tH^{1-\sigma}} \frac{\partial^2 \tilde Q_{k+1}}{\partial b_{ij} \partial b_{pq}} \nabla_m b_{ij} \nabla_m b_{pq} + \sigma |\tA|^2 f + c D \tH^{1+\sigma} \nonumber \\
&\leq 2 \frac{|\nabla \tH||\nabla f|}{\tH} - \frac{1}{c} \frac{|\nabla \tA|^2}{\tH^{2-\sigma}} + c \tH^\sigma + \sigma |\tA|^2 f + cD \tH^{1+\sigma} \label{eqn:evolution-of-f} .
\end{align}

\begin{lemma}\label{lemma:convex-EVOLUTION-LIKE}
There are constants $c = c(S, \Sigma_0, T, n, k, \epsilon, \eta)$ and $C = C(c, p, \sigma, D)$ such that whenever $p > p_0(c, n)$, we have
\begin{align*}
\partial_t \int_{\Sigma_t} f^p_k
&\leq -p^2/3 \int_{\Sigma_t} f^{p-2}_k |\nabla f|^2 - p/c \int_{\Sigma_t} f^{p-1}_k \frac{|\nabla \tA|^2}{\tH^{2-\sigma}} + c\int_{\partial\Sigma_t} f^{p-1}_k \tH^\sigma \\
&\quad + 2p\sigma \int_{A(k, t)} f^p \tH^2 + C \int_{A(k, t)} f^p + C |A(k, t)|  - 1/5 \int_{\Sigma_t} f^p_k \tH^2
\end{align*}
\end{lemma}

\begin{proof}
We have by equation \eqref{eqn:evolution-of-f} (all integrals over $\Sigma_t$ unless stated),
\begin{align*}
\partial_t \int f^p_k
&= p \int f^p_k \Delta f - \int f^p_k H^2 \\
&\leq -p(p-1) \int f^{p-2}_k |\nabla f|^2 + p\int_{\partial\Sigma} f^{p-1}_k |N f| + p^2/3 \int f^{p-2}_k |\nabla f|^2 \\
&\quad + 3c \int f^{p-1}_k \frac{|\nabla \tH|^2}{\tH^{2-\sigma}} - p/c \int f^{p-1}_k \frac{|\nabla \tA|^2}{\tH^{2-\sigma}} + pc \int f^{p-1}_k \tH^\sigma \\
&\quad + p\sigma \int_{A(k, t)} f^p |\tA|^2 + cD \int f^{p-1}_k \tH^{1+\sigma} - \int f^p_k H^2
\end{align*}
provided $p > 2c^2n$.  Here $c = c(S, \Sigma_0, T, n, \epsilon, \eta)$ and $C = C(c, p, \sigma, D)$.  The last term results from
\[
H^2 = \left( \frac{1}{1+n\epsilon} \tH + O(1)\right)^2 \geq \frac{1}{4} \tH^2 + O(1) .
\]
The boundary term is handled by Theorem \ref{theorem:f-boundary-derivative}.  And the other terms are handled by Peter-Paul and/or Remark \ref{remark:handle-low-H-powers}.
\end{proof}


We obtain a (POINCARE-LIKE) equation for $f$.

\begin{lemma}\label{lemma:convex-laplace-inequality}
For $\epsilon < \epsilon_0(n, k, \eta)$ we have on $\spt f_+$ 
\begin{align*}
\frac{\partial \tS_k}{\partial b_{ij}} \nabla_i \nabla_j \tS_{k+1}
&\geq \frac{\partial \tS_k}{\partial b_{ij}} \frac{\partial^2 \tS_{k+1}}{\partial b_{lm} \partial b_{pq}} \nabla_i b_{lm} \nabla_j b_{pq} + \frac{\partial \tS_k}{\partial b_{ij}} \frac{\partial \tS_{k+1}}{\partial b_{lm}} \nabla_l \nabla_m b_{ij} \\
&\quad + \frac{\epsilon}{1+n\epsilon} \left( (n-k) \tS_k \frac{\partial \tS_k}{\partial b_{ij}} - (n-k+1) \tS_{k-1} \frac{\partial \tS_{k+1}}{\partial b_{ij}} \right) \nabla_i\nabla_j \tH \\
&\quad + \frac{1}{2}\eta \tH^2\tS^2_k - cD|\tA|^{2k}(D + |\tA|) - c|\tA|^{2k-1} |\nabla \tA|
\end{align*}
\end{lemma}

\begin{proof}
We follow the proof of Lemma 2.15 and Corollary 2.16 in \cite{huisken-sinestrari:convex-pinching}.  Recall we fixed $D = D_\epsilon$.  From Proposition \ref{prop:tensor-evolution} and Remark \ref{remark:orders} we have $\nabla_p \nabla_q T_{ij} = O(|\tA|^2 + |\nabla \tA|)$.  In particular, 
\[
\nabla_p \nabla_q \tH = (1+n\epsilon) \nabla_p \nabla_q H + O(|\tA|^2 + |\nabla \tA|)
\]
and
\begin{align*}
\nabla_i \nabla_j \bar h_{lm} - \nabla_l \nabla_m \bar h_{ij} &= \bar h_{ij} \bar h_{lr} \bar h_{rm} - \bar h_{lm} \bar h_{ir} \bar h_{rj} + \bar h_{im} \bar h_{lr} \bar h_{rj} - \bar h_{lj} \bar h_{mr} \bar h_{ri} \\
&\quad\quad + O(|\tA|^2 + |\nabla \tA|) .
\end{align*}
We therefore calculate
\begin{align*}
&\frac{\partial \tS_k}{\partial b_{ij}} \frac{\partial \tS_{k+1}}{\partial b_{lm}} \left( \nabla_i \nabla_j b_{lm} - \nabla_l \nabla_m b_{ij} \right) \\
&= \frac{\partial \tS_k}{\partial b_{ij}} \frac{\partial \tS_{k+1}}{\partial b_{lm}} \left[\nabla_i \nabla_j \bar h_{lm} - \nabla_l \nabla_m \bar h_{ij}  \right. \\
&\quad \left. + \frac{\epsilon}{1+n\epsilon} \left( \delta_{lm} \nabla_i\nabla_j \tH - \delta_{ij} \nabla_l \nabla_m \tH + O(|\tA|^2 + |\nabla \tA|) \right) \right] \\
&\geq \frac{\partial \tS_k}{\partial b_{ij}} \frac{\partial \tS_{k+1}}{\partial b_{lm}} \left[\nabla_i \nabla_j \bar h_{lm} - \nabla_l \nabla_m \bar h_{ij} + \frac{\epsilon}{1+n\epsilon} \left( \delta_{lm} \nabla_i\nabla_j \tH - \delta_{ij} \nabla_l \nabla_m \tH \right)\right] \\
&\quad - c|\tA|^{2k+1} - c|\tA|^{2k-1} |\nabla \tA| .
\end{align*}

Choose a frame which diagonalizes $\bar h_{ij}$, and hence $b_{ij}$, then
\begin{align*}
\frac{\partial \tS_k}{\partial b_{ij}} \frac{\tS_{k+1}}{\partial b_{lm}} (\nabla_i \nabla_j \bar h_{lm} - \nabla_l \nabla_m \bar h_{ij})
&= \frac{\partial \tS_k}{\partial \tilde \lambda_i} \frac{\partial \tS_{k+1}}{\partial \tilde \lambda_m} \left[ \blambda_i \blambda_m^2 - \blambda_i^2 \blambda_m + O(|\tA|^2 + |\nabla\tA|) \right] \\
&\geq \frac{\partial \tS_k}{\partial \tlambda_i} \frac{\partial \tS_{k+1}}{\partial \tlambda_m} \left[ \tlambda_i \tlambda_m^2 - \tlambda_i^2\tlambda_m + O(|\tA|^2 + |\nabla \tA|) \right. \\
&\quad\quad + \left(\frac{\epsilon}{1+n\epsilon} \tH + O(D)\right)^2 (\tlambda_m - \tlambda_i) \\
&\quad\quad \left. + \left(\frac{\epsilon}{1+n\epsilon} \tH + O(D)\right)(\tlambda_i^2 - \tlambda_m^2) \right]\\
&\geq \frac{\partial \tS_k}{\partial \tilde \lambda_i} \frac{\partial \tS_{k+1}}{\partial \tilde \lambda_m} \left[ \tlambda_i \tlambda_m^2 - \tlambda_i^2\tlambda_m \right. \\
&\quad\quad \left. + \left(\frac{\epsilon \tH}{1+n\epsilon}\right)^2 (\tlambda_m - \tlambda_i) + \left(\frac{\epsilon \tH}{1+n\epsilon}\right)(\tlambda_i^2 - \tlambda_m^2) \right] \\
&\quad -c D |\tA|^{2k} (D + |\tA|) - c |\tA|^{2k-1} |\nabla \tA|.
\end{align*}

Therefore, by precisely the same arguments at in Lemma 2.15 of \cite{huisken-sinestrari:convex-pinching}, we have for any $\epsilon > 0$
\begin{align*}
\frac{\partial \tS_k}{\partial b_{ij}} \nabla_i \nabla_j \tS_{k+1}
&\geq \frac{\partial \tS_k}{\partial b_{ij}} \frac{\partial^2 \tS_{k+1}}{\partial b_{lm} \partial b_{pq}} \nabla_i b_{lm} \nabla_j b_{pq} \\
&\quad + \frac{\partial \tS_k}{\partial b_{ij}} \frac{\partial \tS_{k+1}}{\partial b_{lm}} \nabla_l \nabla_m b_{ij} \\
&\quad + \frac{\epsilon}{1+n\epsilon} \left( (n-k) \tS_k \frac{\partial \tS_k}{\partial b_{ij}} - (n-k+1) \tS_{k-1} \frac{\partial \tS_{k+1}}{\partial b_{ij}} \right) \nabla_i\nabla_j \tH \\
&\quad - \tH \tS_k \tS_{k+1} + (k+1) \tS^2_k + k((k+1) \tS^2_{k+1} -(k+2) \tS_k\tS_{k+2}) \\
&\quad + \left(\frac{\epsilon \tH}{1+n\epsilon}\right)^2 \left[ (k+1)(n-k+1)\tS_{k+1}\tS_{k-1} - k(n-k) \tS^2_k \right] \\
&\quad + \left(\frac{\epsilon \tH}{1+n\epsilon}\right) \left[ (n-k)\tS_k (\tH \tS_k -(k+1)\tS_{k+1}) \right.\\
&\quad\quad \left. + (n-k+1)\tS_{k-1}((k+2)(\tS_{k+2} - \tH\tS_{k+1})\right] \\
&\quad - cD|\tA|^{2k}(D + |\tA|) - c|\tA|^{2k-1} |\nabla \tA|.
\end{align*}
And the Lemma follows by the same argument as in Corollary 2.16 of \cite{huisken-sinestrari:convex-pinching}.
\end{proof}

\begin{lemma}\label{lemma:convex-POINCARE-LIKE}
There is a constant $c = c(S, \Sigma_0, T, n, \epsilon, \eta, D)$ such that for any $p > 2$ and $\beta > 0$, we have
\begin{align*}
\frac{1}{c} \int_{\Sigma_t} f^p_+ \tH^2 
&\leq (p+p/\beta) \int_{\Sigma_t} f^{p-2}_+ |\nabla f|^2 + (1+\beta p) \int_{\Sigma_t} f^{p-1}_+ \frac{|\nabla \tA|^2}{\tH^{2-\sigma}} + \int_{\Sigma_t} f^p_+ \\
&\quad + \int_{\partial\Sigma_t} f^{p-1}_+\tH^\sigma
\end{align*}
\end{lemma}

\begin{proof}
Fix $\epsilon = \epsilon(\eta, n, k)$ as in Lemma \ref{lemma:convex-laplace-inequality}.  Using inequalities of Remarks \ref{remark:tilde-S-bounded-below} and \ref{remark:derivative-bounds}, we have for $c = c(S, \Sigma_0, T, n, k, \epsilon)$, 
\begin{align*}
\frac{\partial \tS_k}{\partial b_{ij}} \nabla_i\nabla_j f
&\leq -\tH^{\sigma-1} \tS^{-1}_k \frac{\partial \tS_k}{\partial b_{ij}} \nabla_i\nabla_j \tS_{k+1} + c \tH^{k+\sigma-3}|\nabla \tA|^2 + c\tH^{k-2}|\nabla f||\nabla \tA| \\
&\quad + \frac{\partial \tS_k}{\partial b_{ij}} \left[ \tH^{\sigma-1}\tS^{-2}_k\tS_{k+1}\nabla_i\nabla_j \tS_k - (\eta \tH^{\sigma-1} - (\sigma-1) \tH^{-1}f)\nabla_i\nabla_j \tH \right]
\end{align*}

Multiply by $f^p_+ \tH^{-k+1-\sigma}$, integrate, and use Lemma \ref{lemma:convex-laplace-inequality} to obtain
\begin{align*}
\frac{\eta}{2c \epsilon^{k-1}} \int f^p_+ \tH^2 &\leq \\
\frac{\eta}{2} \int \tS_k \tH^{2-k} f^p_+ 
&\leq -\int f^p_+ \tH^{-k+1-\sigma} \frac{\partial \tS_k}{\partial b_{ij}} \nabla_i\nabla_j f \\
&\quad - \int f^p_+ \tH^{-k} \tS^{-1}_k \frac{\partial \tS_k}{\partial b_{ij}} \left\{ -\tS^{-1}_k \tS_{k+1} \nabla_i \nabla_j \tS_k \right. \\
&\quad\quad \left. + \frac{\partial^2 \tS_{k_1}}{\partial b_{lm} \partial b_{pq}} \nabla_i b_{lm} \nabla_j b_{pq} + \frac{\partial \tS_{k+1}}{\partial b_{lm}} \nabla_l \nabla_m b_{ij} \right\} \\
&\quad + \int f^p_+ \tH^{-k} \frac{\partial \tS_k}{\partial b_{ij}} \left( -\eta - \frac{\epsilon}{1+n\epsilon}(n-k) + (\sigma - 1)\tH^{-\sigma}f \right) \nabla_i \nabla_j \tH \\
&\quad + \frac{\epsilon}{1+n\epsilon} (n-k+1)\int f^p_+ \tH^{-k} \tS_{k-1} \tS^{-1}_k \frac{\partial \tS_{k+1}}{\partial b_{ij}} \nabla_i\nabla_j \tH \\
&\quad + cD \int f^p_+ + cD^2 \int f^p_+ \tH + c \int f^p_+ \tH^{-1} |\nabla \tA| \\
&\quad + c \int f^p_+ \tH^{-2} |\nabla \tA|^2 + c\int f^p_+ \tH^{-1-\sigma}|\nabla \tA||\nabla f|
\end{align*}
where $c = c(S, \Sigma_0, T, n, \epsilon)$.  As usual all integrals are over $\Sigma_t$ unless otherwise stated.

Integrate by parts all double covariant derivatives, using Theorem \ref{theorem:f-boundary-derivative} and equation \eqref{eqn:tilde-S-boundary} to handle boundary terms.  After applying remarks \ref{remark:tilde-S-bounded-below} and \ref{remark:derivative-bounds}, we obtain that
\begin{align*}
\frac{\eta}{c}\int f^p_+ \tH^2 
&\leq c\int_{\partial\Sigma_t} f^{p-1}_+f + c \int f^{p-1}_+ \frac{|\nabla \tA|^2}{\tH^{2-\sigma}} + c \int f^p_+ \frac{|\nabla f||\nabla \tA|}{\tH^{1+\sigma}} \\
&\quad + cp \int f^{p-1}_+ \frac{|\nabla f||\nabla \tA|}{\tH} + cp \int f^{p-2}_+|\nabla f|^2 + c \int f^p_+ \frac{|\nabla \tA|}{\tH} \\
&\quad + \frac{1}{\mu} \int f^p_+ \tH^2 + C(c, D, \mu) \int f^p_+
\end{align*}
where $\mu > 0$ is arbitrary.  Set $\mu = 2c/\eta$.  Recalling that $f \leq c \tH^\sigma$, the Lemma follows by using Peter-Paul on the remaining terms.
\end{proof}


In view of Lemma \ref{lemma:f-bounded-implies} and Theorem \ref{theorem:f-is-bounded}, to finish proving Theorem \ref{theorem:convexity-pinching} we merely need to show $f_+$ satisfies $(\star)$ of Section \ref{section:stampacchia}.  In the language of Section \ref{section:stampacchia}, let $\tH$ be itself (the twice-perturbed mean curvature), and $\tilde G = |\nabla \tilde A|$.  Then Lemmas \ref{lemma:convex-EVOLUTION-LIKE} and \ref{lemma:convex-POINCARE-LIKE} imply $f_+$ satisfies $(\star)$.  We are done.

\section{Umbilic pinching when $S = S^n$}
\label{section:umbilic-pinching}

We consider the case when $\Sigma_0$ is strictly convex and $S$ is the sphere $S^n$.  We prove the umbilic pinching Theorem \ref{theorem:umbilic-pinching}.  By Proposition \ref{prop:H-boundary-derivative} we know that $T \leq T_0(\Sigma_0) < \infty$.

Notice in this case $h_{N, X} \equiv 0$ for all $X \in T_p\partial\Sigma$, so the estimates of Lemma \ref{lemma:boundary-derivatives} give us directly boundary conditions on the principle curvatures.  We can therefore work with the unperturbed second fundamental form.  In conjunction with the following Remark we have
\begin{equation}\label{eqn:spherical-boundary-derivative}
NH = H, \quad N |A| = O(H) .
\end{equation}

\begin{remark}\label{remark:convexity-preserved}
By Theorem 9.7 in \cite{stahl:singularity}, there is an $\epsilon = \epsilon(\Sigma_0, n)$ such that
\begin{equation}
h_{ij} \geq \epsilon H g_{ij}
\end{equation}
for all $t \in [0, T)$.  Hence the pointwise estimates of Lemma 2.3 in \cite{huisken:umbilic-pinching} continue to hold in the spherical-free-boundary case.
\end{remark}

Arguing as in \cite{huisken:umbilic-pinching}, to prove Theorem \ref{theorem:umbilic-pinching} it will suffice to show the following Theorems.
\begin{theorem}\label{theorem:umbilic-f-is-bounded}
For any $\eta > 0$, we have
\[
|A|^2 - \frac{1}{n}H^2 \leq \eta H^2 + C(\eta, \Sigma_0, n) .
\]
\end{theorem}

\begin{theorem}\label{theorem:umbilic-gradient-estimate}
For any $\eta > 0$, we have
\[
|\nabla H|^2 \leq \eta H^4 + C(\eta, \Sigma_0, n) .
\]
\end{theorem}

We first prove Theorem \ref{theorem:umbilic-f-is-bounded}.  For $\sigma > 0$ define the function
\begin{equation}
f = \frac{|A|^2 - \frac{1}{n} H^2}{H^{2-\sigma}} = \frac{H^\sigma}{2n} \sum_{i, j} \frac{(\lambda_i - \lambda_j)^2}{H^2} .
\end{equation}
By Remark \ref{remark:convexity-preserved} and equations \ref{eqn:spherical-boundary-derivative}, we have
\begin{equation}\label{eqn:umbilic-boundary-derivative}
f = O(H^\sigma), \quad N f = O(H^\sigma) .
\end{equation}

Clearly to prove Theorem \ref{theorem:umbilic-f-is-bounded} it suffices to show $f$ is bounded in spacetime for some choice of $\sigma > 0$.  We shall demonstrate in the next two Lemmas that $f$ satisfies the (EVOLUTION-LIKE) and (POINCARE-LIKE) equations of Section \ref{section:stampacchia}.

\begin{lemma}\label{lemma:umbilic-integrate-laplace}
There is a constant $c = c(n, \epsilon)$ such that for every $\eta > 0$ we have
\begin{align*}
\frac{1}{c} \int_{\Sigma_t} f^p H^2
&\leq (\eta p + 1) \int_{\Sigma_t} \frac{|\nabla H|^2}{H^{2-\sigma}} f^{p-1} + \frac{p}{\eta} \int_{\Sigma_t} |\nabla f|^2 f^{p-2} \\
&\quad + \int_{\partial\Sigma} f^{p-1} H^\sigma
\end{align*}
\end{lemma}

\begin{proof}
We follow the proof of Lemma 5.4 in \cite{huisken:umbilic-pinching}.  In consideration of Remark \ref{remark:convexity-preserved}, we have
\begin{align}
2n\epsilon^2 f^p H^2
&\leq \frac{2}{H^{2-\sigma}} f^{p-1} Z \nonumber \\
&\leq f^{p-1} \Delta f \label{eqn:umbilic-term2} \\
&\quad - \frac{2}{H^{2-\sigma}} f^{p-1} <h^0_{ij}, \nabla_i \nabla_j H> \label{eqn:umbilic-term3} \\
&\quad + \frac{2(1-\sigma)}{H} f^{p-1} <\nabla H, \nabla f> \nonumber \\
&\quad + \frac{2-\sigma}{H} f^p \Delta H . \label{eqn:umbilic-term5}
\end{align}
Here $h^0_{ij}$ is the trace-free second fundamental form, and $Z = H \tr(A^3) - |A|^4$.  We integrate the above relation, and integrate by parts terms \eqref{eqn:umbilic-term2}, \eqref{eqn:umbilic-term3} and \eqref{eqn:umbilic-term5}.  The resulting interior terms are handled by Peter-Paul, and the inequality $|h^0_{ij}| \leq f H^{2-\sigma} \leq H^2$.  To handle the boundary term use \eqref{eqn:spherical-boundary-derivative} and \eqref{eqn:umbilic-boundary-derivative}, and that $h_{N, X}$ vanishes when $X \in T_p \partial\Sigma$.
\end{proof}

Recall that $f$ satisfies the evolution inequality
\begin{equation}\label{eqn:umbilic-evolution-equation}
\partial_t f \leq \Delta f + \frac{2(1-\sigma)}{H} <\nabla H, \nabla f> - \epsilon^2 \frac{|\nabla H|^2}{H^{2-\sigma}} + \sigma |A|^2 f
\end{equation}

\begin{lemma}\label{lemma:umbilic-integrate-evolution}
We have, for $c = c(n, \epsilon)$, 
\begin{align*}
\partial_t \int_{\Sigma_t} f^p_k
&\leq -p^2/3 \int_{\Sigma_t} |\nabla f|^2 f^{p-2}_k - p/c \int_{\Sigma_t} \frac{|\nabla H|^2}{H^{2-\sigma}} f^{p-1}_k \\
&\quad + (\sigma p - 1) \int_{\Sigma_t} H^2 f^p + c p \int_{\partial\Sigma_t} f^{p-1}_k H^\sigma
\end{align*}
\end{lemma}

\begin{proof}
Follows directly by \eqref{eqn:umbilic-evolution-equation}, and Proposition \ref{prop:general-evolution-equations}.  Use Peter-Paul to handle the inner product term, and equation \eqref{eqn:umbilic-boundary-derivative} to handle the boundary term obtained upon integration by parts.
\end{proof}

In view of Lemmas \ref{lemma:umbilic-integrate-laplace} and \ref{lemma:umbilic-integrate-evolution}, we can take $\tilde H = H$ and $\tilde G = |\nabla H|$ in Section \ref{section:stampacchia}.  Theorem \ref{theorem:umbilic-f-is-bounded} now follows by Theorem \ref{theorem:f-is-bounded}.

We prove Theorem \ref{theorem:umbilic-gradient-estimate}.  We cannot obtain a boundary condition on $|\nabla H|^2$, and therefore we work instead with the quantity $|\nabla H - HV|^2$.  Here we define $V \equiv \bar V^T$, with $\bar V$ a fixed vector field on $R^{n+1}$ extending $\nu_S$, such that $\bar\nabla_{\nu_S} \bar V \equiv 0$ on $S^n$.

Fix $\eta > 0$, and choose $\sigma = \sigma(\Sigma_0, n)$ and $C_0 = C_0(\Sigma_0, n)$ so that $f \leq C_0$.  Define
\begin{align*}
g &= \frac{|\nabla H - H V|^2}{H} + b H (|A|^2 - H^2/n) + ba |A|^2 - \eta H^3 + D
\end{align*}
for $a, b, D$ positive constants to be determined, depending only on $(\eta, \Sigma_0, n)$.  We will show $g$ is bounded in spacetime, which clearly suffices to prove Theorem \ref{theorem:umbilic-gradient-estimate} since
\[
|\nabla H - HV|^2 \geq \frac{1}{2}|\nabla H|^2 - c_n H^2 .
\]

\begin{lemma}\label{lemma:umbilic-gradient-evolution}
We have the evolution equations
\begin{align*}
&(\partial_t - \Delta) |\nabla H - H V|^2 \leq c H^2 |\nabla A|^2 + c H^4 - 2|\nabla (\nabla H - HV)|^2 \\
&(\partial_t - \Delta) \frac{|\nabla H - HV|^2}{H} \leq c H |\nabla A|^2 + c H^3
\end{align*}
for $c = c(\Sigma_0, n)$.
\end{lemma}

\begin{proof}
By direct calculation we have that
\begin{align*}
\nabla_i V &= (\bar\nabla_i \bar V)^T = O(1)\\
\Delta V &= -H (\bar\nabla_\nu \bar V)^T + (\tr_{T\Sigma} \bar\nabla^2 \bar V)^T\\
&= O(H + 1) ,
\end{align*}
and 
\begin{align*}
\partial_t V_i
&= \partial_t <\bar V, \partial_i> \\
&= -H \bar\nabla_\nu \bar V_i - H h_{ij} V_j - \partial_i H <\bar V, \nu> .
\end{align*}

We have:
\begin{align*}
\frac{1}{2} \Delta |\nabla H - HV|^2
&= \left\{ \nabla_i \Delta H + \nabla_j H (H h_{ij} - h_{ik}h_{kj}) \right. \\
&\quad\quad \left. - \Delta H V_i + O(H^2 + H + |\nabla H|) \right\} (\nabla_i H - H V_i) \\
&\quad + |\nabla(\nabla H - H V)|^2
\end{align*}
and
\begin{align*}
\frac{1}{2} \partial_t |\nabla H - H V|^2 
&= \left\{ H h_{ij} \nabla_j H + \nabla_i (\Delta H + |A|^2H) - (\Delta H + |A|^2 H) V_i  \right. \\
&\quad \quad \left. + O(|\nabla H| + H^2) \right\} (\nabla_i H - H V_i) . 
\end{align*}
The first equation follows directly, recalling that $H$ is non-increasing.  To prove the second formula, use equation \eqref{eqn:useful-formula} to obtain
\begin{align*}
(\partial_t - \Delta) \frac{|\nabla H - HV|^2}{H} 
&\leq \frac{c H^2 |\nabla H|^2 + cH^4 - 2|\nabla (\nabla H - HV)|^2}{H} - \frac{|\nabla H - HV|^2 |A|^2}{H}\\
&\quad + \frac{4|\nabla H - HV|}{H^2} |\nabla |\nabla H - HV|||\nabla H| - 2\frac{|\nabla H - HV|^2}{H^3} |\nabla H|^2 \\
&\leq c H |\nabla H|^2 + c H^3 . \qedhere
\end{align*}
\end{proof}

\begin{lemma}\label{lemma:umbilic-gradient-boundary}
At any point on the boundary $\partial\Sigma$, we have
\begin{align*}
&N |\nabla H - HV|^2 = 0 \\
&N ( |A|^2 - H^2/n) \leq c_n H \sqrt{|A|^2 - H^2/n} .
\end{align*}
\end{lemma}
\begin{proof}
Choose orthonormal coordinates at $p$, such that $\partial_1 \equiv N$ along $\partial\Sigma$, the integral curves of $\partial_1$ are geodesics.  We calculate
\begin{align*}
\frac{1}{2} \partial_1 |\nabla H - H V|^2 
&= -\sum_{i, j > 1} g^{ij} (\partial_i H - H V_i)(\partial_j H - H V_j) \\
&\quad + \sum_i (\partial_1 \partial_i H - \partial_1 (H V_i))(\partial_i H - H V_i) \\
&= -|\nabla H - H V|^2 + \sum_{i > 1} (\partial_i \partial_1 H)(\partial_i H) \\
&= -|\nabla H - HV|^2 + |\nabla H|^2 - |\nabla_1 H|^2 .
\end{align*}

We prove the second formula.  Write $\lambda_i$ for the principle curvatures, and $\lambda_N$ for the curvature in direction $N$.  Using Lemma \ref{lemma:boundary-derivatives}, we obtain
\begin{align*}
\frac{1}{2}\partial_1 (|A|^2 - H^2/n)
&= 3H \lambda_N - n\lambda_N^2 - |A|^2 - H^2/n \\
&= (|A|^2 - H^2/n) + 3 H \lambda_N - n\lambda_N^2 - 2|A|^2 .
\end{align*}

We calculate
\begin{align*}
3 H \lambda_N - n\lambda_N^2 - 2|A|^2
&= \sum_i (3\lambda_N \lambda_i - \lambda_N^2 - 2 \lambda_i^2) \\
&= \sum_i (\lambda_i - \lambda_N)(\lambda_N - 2\lambda_i) \\
&= -2\sum_i (\lambda_i - \lambda_N)^2 - \lambda_N \sum_i (\lambda_i - \lambda_N) \\
&\leq \frac{-1}{2(n-1)} \sum_{i, j} (\lambda_i - \lambda_j)^2 + H\sqrt{n\sum_{i,j} (\lambda_i - \lambda_j)^2} \\
&= -\frac{n}{n-1} (|A|^2 - H^2/n) + \sqrt{2} n H \sqrt{ |A|^2 - H^2/n } . \qedhere
\end{align*}
\end{proof}

Using Lemma \ref{lemma:umbilic-gradient-boundary} and equation \eqref{eqn:spherical-boundary-derivative}, we then have
\begin{align*}
N g &\leq bH (|A|^2 - H^2/n) + c_n b H^2 \sqrt{|A|^2 - H^2/n} + c_n ba |A|^2 - 3\eta H^3\\
& \leq (g - D) + c_n b H^2 \sqrt{ C_0 H^{2-\sigma}} + c_n ba H^2 - 2\eta H^3 \\
& \leq g
\end{align*}
provided $D = D(\eta, C_0, \sigma, a, b, n)$ is sufficiently big.

By Theorem \ref{theorem:umbilic-f-is-bounded} (see Lemma 6.5 of \cite{huisken:umbilic-pinching}), we can choose $a = a(C_0, \sigma, n)$ sufficiently large so that,
\begin{equation}\label{eqn:random-evolution-equation}
(\partial_t - \Delta)(|A|^2 - H^2/n) \leq -c_n H |\nabla A|^2 + a |\nabla A|^2 + 3 H^3 (|A|^2 - H^2/n) .
\end{equation}

Using Lemma \ref{lemma:umbilic-gradient-evolution} and equation \eqref{eqn:random-evolution-equation}, we have for $b = b(\Sigma_0, n)$ sufficiently large
\begin{align*}
(\partial_t - \Delta) g 
&\leq c H |\nabla A|^2 + c H^3 + 6 n H |\nabla A|^2 - b c_n H |\nabla A|^2 \\
&\quad + 2 ba H^4 + 3 b H^3(|A|^2 - H^2/n) - 3\frac{\eta}{n} H^5 \\
&\leq 2baH^4 + 3b C_0 H^{5-\sigma} - 3 \frac{\eta}{n} H^5 \\
&\leq C(\eta, \sigma, C_0, \Sigma_0, n)
\end{align*}
using that $c = c(\Sigma_0, n)$.

Take $\phi$ the cutoff function of Section \ref{section:controlling-A}, with constants chosen so that
\[
(\partial_t - \Delta) \phi \leq 0, \quad N \phi \leq -\phi
\]

Then for $a, b, D$ chosen as above, we have
\[
N (g \phi) \leq 0
\]
and
\[
(\partial_t - \Delta) g\phi \leq C \phi + 2<\nabla g, \nabla \phi> .
\]
Since $\phi$ is uniformly bounded in time, we deduce that $\max_{\Sigma_t} g\phi$ increases at worst linearly, and therefore
\[
g \leq \tilde C(\eta, \Sigma_0, n) .
\]

This completes the proof of the umbilic pinching Theorem \ref{theorem:umbilic-pinching}.

\bibliographystyle{alpha}

\begin{thebibliography}{Sta96b}

\bibitem[Buc05]{buckland}
J.~Buckland.
\newblock Mean curvature flow with free boundary on smooth hypersurfaces.
\newblock {\em J. reine angew. Math.}, 586:71--90, 2005.

\bibitem[HS99a]{huisken-sinestrari:convex-pinching}
G.~Huisken and C.~Sinestrari.
\newblock Convexity estimates for mean curvature flow and singularities of mean
  convex surfaces.
\newblock {\em Acta Math.}, 183:45--70, 1999.

\bibitem[HS99b]{huisken-sinestrari:scalar-pinching}
G.~Huisken and C.~Sinestrari.
\newblock Mean curvature flow singularities for mean convex surfaces.
\newblock {\em Calc. Var.}, 8:1--14, 1999.

\bibitem[Hui84]{huisken:umbilic-pinching}
G.~Huisken.
\newblock Flow by mean curvature of convex surfaces into spheres.
\newblock {\em J. Differential Geometry}, 20:237--266, 1984.

\bibitem[MS73]{michael-simon}
J.~H. Michael and L.~M. Simon.
\newblock Sobolev and mean-value inequalities on generalized submanifolds on
  $R^n$.
\newblock {\em Comm. Pure Appl. Math.}, 26:361--379, 1973.

\bibitem[Sta96a]{stahl:singularity}
A.~Stahl.
\newblock Convergence of solutions to the mean curvauture flow with a neumann
  boundary condition.
\newblock {\em Calc. Var.}, 4:421--441, 1996.

\bibitem[Sta96b]{stahl:regularity}
A.~Stahl.
\newblock Regularity estimates for solutions to the mean curvature flow with a
  neumann boundary condition.
\newblock {\em Calc. Var.}, 4:385--407, 1996.

\end{thebibliography}

\end{document}